\documentclass[11pt]{amsart}
\usepackage{lmodern}
\usepackage{amsmath, amsthm, amssymb, amsfonts}
\usepackage[normalem]{ulem}
\usepackage{hyperref}
\usepackage{mathtools}
\usepackage{stmaryrd}

\hypersetup{
    colorlinks=true,
    linkcolor=black,
    filecolor=magenta,      
    urlcolor=cyan,
    citecolor=magenta,
    pdftitle={Overleaf Example},
    pdfpagemode=FullScreen,
    }

\usepackage{mathrsfs}
\usepackage{graphicx}

\usepackage{verbatim} 
\usepackage{longtable}

\usepackage{mathtools}

\usepackage{tikz}
\usetikzlibrary{decorations.pathmorphing}
\tikzset{snake it/.style={decorate, decoration=snake}}

\usepackage{caption}

\usepackage{tikz-cd}
\usetikzlibrary{arrows}

\newsavebox{\pullback}
\sbox\pullback{
\begin{tikzpicture}%
\draw (0,0) -- (1ex,0ex);%
\draw (1ex,0ex) -- (1ex,1ex);%
\end{tikzpicture}}

\theoremstyle{plain}
\newtheorem{thm}{Theorem}[section]
\newtheorem{cor}[thm]{Corollary}
\newtheorem{lem}[thm]{Lemma}
\newtheorem{prop}[thm]{Proposition}
\newtheorem{conj}[thm]{Conjecture}
\newtheorem{question}[thm]{Question}

\theoremstyle{definition}
\newtheorem{defn}[thm]{Definition}

\theoremstyle{remark}
\newtheorem{rmk}[thm]{Remark}

\newcommand{\gr}{{\mathrm{gr}}}
\newcommand{\Sp}{{\mathrm{Sp}}}

\newcommand{\BC}{{\mathbb{C}}}
\newcommand{\BD}{{\mathbb{D}}}

\newcommand{\BG}{{\mathbb{G}}}

\newcommand{\BK}{{\mathbb{K}}}
\newcommand{\BL}{{\mathbb{L}}}

\newcommand{\BP}{{\mathbb{P}}}
\newcommand{\BQ}{{\mathbb{Q}}}

\newcommand{\BZ}{{\mathbb{Z}}}

\newcommand{\mr}{\mathsf{MR}}

\newcommand{\CB}{{\mathcal B}}

\newcommand{\CH}{{\mathcal H}}

\newcommand{\CL}{{\mathcal L}}

\newcommand{\CO}{{\mathcal O}}

\newcommand{\CU}{{\mathcal U}}
\newcommand{\CV}{{\mathcal V}}

\newcommand{\CZ}{{\mathcal Z}}

\newcommand{\Fe}{{\mathfrak{e}}}
\newcommand{\Ff}{{\mathfrak{f}}}

\newcommand{\Fh}{{\mathfrak{h}}}

\newcommand{\FM}{{\mathfrak{M}}}
\newcommand{\FN}{{\mathfrak{N}}}

\newcommand{\PPic}{{\mathfrak{P}\textup{ic}}}

\newcommand{\pt}{{\mathsf{pt}}}
\newcommand{\ch}{{\mathrm{ch}}}

\DeclareFontFamily{OT1}{rsfs}{}
\DeclareFontShape{OT1}{rsfs}{n}{it}{<-> rsfs10}{}
\DeclareMathAlphabet{\curly}{OT1}{rsfs}{n}{it}

\newcommand\Ext{\operatorname{Ext}}
\newcommand\Hom{\operatorname{Hom}}
\newcommand\End{\operatorname{End}}
\newcommand\Aut{\operatorname{Aut}}

\newcommand{\rk}{\mathrm{rk}}

\newcommand{\Pic}{\mathop{\rm Pic}\nolimits}

\usepackage{tikz}
\usepackage{lmodern}
\usetikzlibrary{decorations.pathmorphing}

\begin{document}
\title[Chern filtration for moduli of bundles on curves]{On the Chern filtration for the moduli of bundles\\ on curves}
\date{\today}

\newcommand\blfootnote[1]{%
  \begingroup
  \renewcommand\thefootnote{}\footnote{#1}%
  \addtocounter{footnote}{-1}%
  \endgroup
}

\author[W. Lim]{Woonam Lim}
\address{Yonsei University, Department of Mathematics}
\email{woonamlim@yonsei.ac.kr}

\author[M. Moreira]{Miguel Moreira}
\address{Massachusetts Institute of Technology, Department of Mathematics}
\email{miguel@mit.edu}

\author[W. Pi]{Weite Pi}
\address{Yale University, Department of Mathematics}
\email{weite.pi@yale.edu}


\begin{abstract}

We introduce and study the Chern filtration on the cohomology of the moduli of bundles on curves. This can be viewed as a natural cohomological invariant defined via tautological classes that interpolates between additive Betti numbers and the multiplicative ring structure. In the rank two case, we fully compute the Chern filtration for moduli of stable bundles and all intermediate stacks in the Harder--Narasimhan stratification. We observe a curious symmetry of the Chern filtration on the moduli of rank two stable bundles, and construct $\mathfrak{sl}_2$-actions that categorify this symmetry. Our study of the Chern filtration is motivated by the $P=C$ phenomena in several related geometries.

\end{abstract}

\maketitle

\setcounter{tocdepth}{1} 

\tableofcontents
\setcounter{section}{-1}

\section{Introduction}

Let $\Sigma$ be a complex smooth projective curve of genus $g \geq 2$, and let $r$ and $d$ be two coprime integers. We denote by $N_{r,d} = N_{r,\Lambda}$ the moduli space of stable rank $r$ and degree $d$ vector bundles on $\Sigma$, with fixed determinant $\Lambda \in \mathrm{Pic}^d(\Sigma)$. It is a smooth projective variety of (complex) dimension $(r^2-1)(g-1)$, whose underlying real manifold structure depends on the curve only through its genus $g$ due to the Narasimhan--Seshadri theorem \cite{NS}.

The geometry and topology of the moduli space $N_{r,d}$ have been extensively studied for decades, dating back to Atiyah--Bott's celebrated work \cite{AB}. In particular, the cohomological aspects of $N_{r,d}$ are well-understood: it is known that the integral cohomology $H^*(N_{r,d}, \mathbb{Z})$ has no torsion and is tautologically generated \cite{AB, Beau}, closed formulas for the Betti numbers are obtained \cite{Zag_recursion} by solving the recursions in \cite{HN, AB}, and a set of geometric relations, originally due to Mumford and later generalized by Earl--Kirwan, are complete \cite{Kir, EK}. Much more work has been done on tautological relations and intersection numbers on $N_{r,d}$; see for example \cite{Tha1, Tha2, Zag_rank2, King-Newstead, EK-pontryagin, Teleman-Woodward}.

In the present paper, we study the cohomology of $N_{r,d}$ from a new perspective, by introducing a natural filtration called the \textit{Chern filtration}. This filtration defines two new invariants attached to $N_{r,d}$: the dimensions of the graded pieces of the filtration, and the associated graded ring. The Chern filtration is defined via tautological classes, so these are not purely topological invariants but depend on the structure of $N_{r,d}$ as a moduli space. They refine the Betti numbers and can be determined from full knowledge of the relations among tautological generators in the cohomology ring of $N_{r,d}$, so we have the following sequence of specializations:

\smallskip

\begin{center}
    \begin{tikzcd}[row sep=0.2cm, column sep=0.7cm]
 \textup{\small Betti numbers} &
 \textup{\small refined Betti numbers} &
 \textup{\small associated graded ring} &
 \textup{\small cohomology ring}\\
\dim H^i(N_{r,d})&
\dim C_j H^{i+j}(N_{r,d})\arrow[l, rightsquigarrow] &
\gr^C_\bullet H^\ast(N_{r,d}) \arrow[l, rightsquigarrow]  &
H^\ast(N_{r,d})\arrow[l, rightsquigarrow]
    \end{tikzcd}
\end{center}

\smallskip

Some of the main results in this paper are a complete determination of these new invariants in rank~two: the refined Betti numbers are given in Theorem \ref{thm: omegaM21}, and we obtain two explicit descriptions of the associated graded ring in Section \ref{subsec: relationsIgr} and in Corollary \ref{cor: f determines I gr}. 

The study of the Chern filtration is greatly motivated by analogous filtrations on moduli spaces of Higgs bundles or 1-dimensional sheaves on del Pezzo surfaces, where ``$P=C$ phenomena'' predicts strong properties for the Chern filtration in such geometries.


\smallskip

We work with cohomology groups with rational coefficients from now on.

\subsection{Tautological classes and Chern filtrations} 
Under the assumption $\mathrm{gcd}(r,d)=1$, there exists a universal rank $r$ vector bundle $\mathbb{U}$ over $N_{r,d} \times \Sigma$ which is  unique up to tensoring by line bundles from $N_{r,d}$.
Fix a basis $\{\mathbf{1}_\Sigma, e_1,\ldots, e_{2g}, \mathsf{pt}\}$ of $H^*(\Sigma)$, with the convention that $e_i\cdot e_{i+g} =\mathsf{pt}$ for $1\leq i \leq g$. Taking Künneth components of the Chern classes of a ``normalized'' $\mathbb{U}$ with respect to this basis gives tautological classes in $H^*(N_{r,d})$. More precisely, we write
\begin{equation}
\label{taut}
c_k(\mathbb{U}) = \alpha_k \otimes \mathsf{pt} + \sum_{i=1}^{2g} \psi_{k,i}\otimes e_i + \beta_k\otimes \mathbf{1}_\Sigma \in H^*(N_{r,d})\otimes H^*(\Sigma).
\end{equation}

The normalizing condition means that we impose $\beta_1 = 0$, which can be achieved by taking an arbitrary $\mathbb{U}$ and tensoring it by a uniquely determined rational line bundle over $N_{r,d}$. We refer to Section~\ref{subsec: descendent algebra} for an equivalent but more canonical 
definition of the classes $\{\alpha_k,\, \beta_k,\,\psi_{k,i}\}$ using the language of descendent algebra and stacks.\footnote{There are various conventions on defining the tautological classes $\{\alpha_k,\, \beta_k,\,\psi_{k,i}\}$. For example, our definition agrees with \cite{jeffreykirwan} but differs slightly from \cite{AB, Kir}; see also Remark \ref{rmk: taut class from End}.}

By \cite{AB, Beau}, the classes $\{\alpha_k,\, \beta_k,\,\psi_{k,i}\}_{2\leq k \leq r}$ generate the rational cohomology ring $H^*(N_{r,d})$. We say that the tautological classes $\alpha_k,\, \beta_k,\,\psi_{k,i}$, defined as the Kunneth components of the $k$-th Chern class of $\mathbb{U}$ in \eqref{taut}, have \textit{Chern degree $k$}. The total Chern degree of a monomial in tautological classes is the sum of the Chern degrees of the respective tautological classes.

\begin{defn}
\label{def: chern}
    The Chern filtration $C_\bullet H^*(N_{r,d})$ is the increasing filtration whose $k$-th step $C_k H^*(N_{r,d})$ is spanned by all monomials in $\{\alpha_s,\, \beta_s,\,\psi_{s,i}\}_{2\leq s \leq r}$ with total Chern degree $\leq k$.
\end{defn}

Since the top cohomology $H^\mathrm{top}(N_{r,d}) = H^{(r^2-1)(2g-2)}(N_{r,d})$ is one-dimensional, there exists a unique integer $\ell$, which we call the \textit{top Chern degree}, such that 
\[
0 = C_{\ell-1} H^{\mathrm{top}}(N_{r,d})\subset C_\ell H^{\mathrm{top}}(N_{r,d}) = H^{\mathrm{top}}(N_{r,d}).
\]

Our first result on the Chern filtration is the following.\footnote{This result is implicit in \cite{EK-pontryagin} although the authors did not introduce the Chern filtration; see Section \ref{sec: 1.1}.}

\begin{thm}
\label{thm: topcherndegree}
    The top Chern degree of $N_{r,d}$ equals $(r+2)(r-1)(g-1)$.
\end{thm}

Note that the trivial bound for the top Chern degree is $(r^2-1)(2g-2)$. The fact that the actual value is around half of this bound is due to the presence of tautological relations in $H^*(N_{r,d})$
, which decrease the top Chern degree.

\subsection{
Main results and conjectures}\label{subsec: conjecturesresults}
We denote the associated graded of the Chern filtration by \[\gr^C_i H^*(N_{r,d})\coloneqq C_i H^*(N_{r,d})/C_{i-1} H^*(N_{r,d})\,\]
and let
\[\gr^C_\bullet H^*(N_{r,d})\coloneqq \bigoplus_{i\geq 0}\gr^C_i H^*(N_{r,d})\,;\]
note that $\gr^C_\bullet H^*(N_{r,d})$ is naturally a bigraded algebra, with a cohomological degree and a Chern degree. Define the \textit{refined Poincaré polynomial}
\[
\Omega(N_{r,d}, q,t) \coloneqq \sum_{i,j \geq 0} \dim \mathrm{gr}^C_i H^{i+j}(N_{r,d}) q^i t^j,
\]
as well as the shifted version (by half of top Chern and cohomological degrees)
\[
\overline{\Omega}(N_{r,d}, q, t) \coloneqq q^{-\frac{1}{2}(r+2)(r-1)(g-1)} t^{-\frac{1}{2}r(r-1)(g-1)}\cdot \Omega(N_{r,d}, q,t).
\]

A motivating question for us to study the Chern filtration is whether it has the ``curious'' symmetry
    \begin{equation}
    \label{eq: main}
        \overline{\Omega}(N_{r,d}, q,t) = \overline{\Omega}(N_{r,d}, q^{-1}, t^{-1})\,,
    \end{equation}
    which is originally inspired by the $P=C$ phenomenon; see Section \ref{sec: P=C}. The main result of this paper consists of a series of structural theorems on the Chern filtration in the rank two case; 
    in particular, we give a positive answer to this question when $r=2$.
    

\begin{thm}
\label{thm: main}
    The refined Poincaré polynomial for $N_{2,1}$ is
    \begin{equation*}
    \Omega(N_{2,1},q,t) = \frac{(1+q^2t)^{2g}-q^{2g}(1+t)^{2g}}{(1-q^2)(1-q^2t^2)}.
    \end{equation*}
    In particular, the symmetry \eqref{eq: main} holds when $r=2$.
\end{thm}


To state our next result, we denote by $\FN_{2,1}^{\leq d}$ the stack of vector bundles on $\Sigma$ with rank two and fixed determinant $\Lambda$ of degree $1$ whose maximal destablizing line bundle has degree $\leq d$. This is an open substack of the stack $\mathfrak{N}_{2,1}$ of rank two bundles with fixed determinant $\Lambda$; see Section \ref{subsec: stacks} for more discussions.\footnote{In particular, the stack $\FN_{2,1}^{\leq 0}$ is a ${\mu_2}$-gerbe over $N_{2,1}$ and their cohomology rings are isomorphic.} The Chern filtration and refined Poincaré series can be defined for $\mathfrak{N}_{2,1}^{\leq d}$ in a parallel way.

\begin{thm}\label{thm: intermediatestacks}
    We have
    \begin{align*}
    \Omega(\mathfrak{N}^{\leq d}_{2,1}, q, t) &=\frac{(1+q^2t)^{2g}-q^{2g+4d}(1+t)^{2g}}{(1-q^2)(1-q^2t^2)}\\
    &=\Omega(N_{2,1}, q, t) + \sum_{k=1}^d q^{2g+4k-4} \frac{(1+t)^{2g}(1+q^2)}{1-q^2t^2}\,.
    \end{align*}
\end{thm}

Although Theorems \ref{thm: main} and \ref{thm: intermediatestacks} look similar in nature, our proof of Theorem~\ref{thm: intermediatestacks} requires establishing Theorem \ref{thm: main} first and analyzing the interaction between the Chern filtration and Harder--Narasimhan stratifications for $\mathfrak{N}_{2,1}$. The proofs of the two theorems will be given in Sections~\ref{sec:refinedpoincareN21} and \ref{sec: intermediatestacks}, respectively.

\smallskip

Furthermore, we construct explicit operators and $\mathfrak{sl}_2$-actions on $\gr^C_\bullet H^*(\FN^{\leq d}_{2,1})$ in Section~\ref{sec: sl2}; when $d=0$, the corresponding $\mathfrak{sl}_2$-triple categorifies the symmetry \eqref{eq: main} for $r=2$. 


\begin{thm}
\label{thm: sl2}
    There are explicit $\mathfrak{sl}_2$-triples $(\Fe^d, \Fh^d, \Ff^d)$ acting on $\gr^C_\bullet H^*(\FN_{2,1}^{\leq d})$ for each $d\geq 0$, such that $\mathfrak h^d$ is a shift (depending on $d$) of the Chern grading operator. Moreover, when $d=0$, $\Fe^0, \Ff^0$ are self-adjoint and $\Fh^0$ is anti-self-adjoint with respect to the associated graded Poincaré pairing \eqref{eq: gradedpairing}.
\end{thm}

The above $\mathfrak{sl}_2$-actions are first defined on the descendent algebra $\BD$, which is the free super-commutative algebra generated by symbols $\alpha_k, \beta_k, \psi_{k, i}$ (see Definition \ref{def: descendents}),
and then shown to descend via the surjection $\BD\twoheadrightarrow \gr^C_\bullet H^*(\FN^{\leq d}_{2,1})$. 


The $\mathfrak{sl}_2$-action with $d=0$ also gives a beautiful description of the associated graded ideal of tautological relations, defined as $I^\gr\coloneqq\ker(\BD\twoheadrightarrow \gr^C_\bullet H^*(N_{2,1}))$. Let $C_{>4g-4}\BD \subset \mathbb{D}$ be the subspace of descendents with Chern degree greater than $4g-4$. Note that Theorem \ref{thm: topcherndegree} implies that $C_{> 4g-4}\mathbb{D} \subset I^\mathrm{gr}$.


\begin{cor}[=\,Corollary \ref{cor: f determines I gr}]
   The associated graded ideal $I^\gr\subset \BD$ is the smallest $\mathfrak{sl}_2$-invariant ideal containing $C_{>4g-4}\BD$.
\end{cor}
    

In the rank three case, we propose a closed formula for $\Omega(N_{3,1}, q, t)$ that satisfies the symmetry~\eqref{eq: main}; it is verified for $g=2,3$ by explicit computations of the cohomology rings. 

\begin{conj}
    The refined Poincaré polynomial for $N_{3,1}$ is
    \begin{align*}
    \Omega(N_{3,1}, q,t) = \frac{1}{(1 - q^2) (1 - 
     q^2 t^2) (1 - q^3 t) (1 - q^3 t^3)} \Big((1 + q^2 t)^{2 g}(1 + q^3 t^2)^{2 g} & \\- 
   q^{4 g - 2} \frac{(1 + q t) (1 - q^3 t)}{1 - t^2} (1 + t)^{2 g} (1 + q t^2)^{2 g} & \\+
   q^{4 g - 2} t^{2g}  \frac{(1 - q^2) (1 + 
       q t + q^2 t^2)}{1-t^2}(1 + q)^{2 g} (1 + t)^{2 g}\Big).&
    \end{align*}
\end{conj}

In general, we ask the following question:

\begin{question} \label{question: symmetry}
    Does the symmetry \eqref{eq: main} hold for $d = 1$ and arbitrary ranks?
\end{question}

This is equivalent to the statement that the intersection pairing
\begin{equation}
\label{eq: gradedpairing}
\mathrm{gr}^C_\bullet H^*(N_{r,1}) \otimes \mathrm{gr}^C_{\mathrm{top}-\bullet}H^*(N_{r,1}) \longrightarrow \mathbb{Q}
\end{equation}
on the associated graded ring is perfect, where $\mathrm{top}$ is the top Chern degree; 
see Corollary~\ref{cor: 1.4} and the discussion afterward. Note that we do \textit{not} expect the symmetry to hold for arbitrary degrees $d$ coprime to the rank: there is a clear topological obstruction to this, which we discuss in Remark \ref{rmk: nosymmetry}. To further investigate Question \ref{question: symmetry}, it would be useful to compute the cohomology rings $H^*(N_{5,d})$ for $d=1,2$ using the geometric relations in \cite{EK}. 


\smallskip

We also propose a conjectural formula for the {$t=-1$ specialization} 
of the refined Poincaré polynomial for general ranks; see Section \ref{subsec: chernandomega} for some motivation and discussions.

\begin{conj}\label{conj: t=-1}
    We have
    \[
    \Omega(N_{r,d}, q, -1)=\prod_{k=2}^r(1-(-q)^k)^{2g-2}.
    \]
\end{conj}

Finally, we point out that the Chern filtration can be defined on the cohomology of moduli spaces of $G$-bundles for other groups $G$. A natural question is to determine the top Chern degree (cf. Theorem \ref{thm: topcherndegree}) for $G$-bundles.

\subsection{Relation to the $P=C$ phenomenon}
\label{sec: P=C} Our study of the Chern filtration for $N_{r,d}$ is largely motivated by the $P=C$ phenomenon in several related geometries, e.g., the moduli space of Higgs bundles on curves, and the moduli space of 1-dimensional sheaves on del Pezzo surfaces (cf. \cite{KPS, KLMP}). Roughly speaking, the $P=C$ phenomenon predicts a match of the Chern filtration on these moduli spaces with the \textit{perverse filtration} induced by certain abelian fibrations. We refer to \cite[Section 1.4]{PSSZ} for a recent account on the history and development of this phenomenon. The $P=C$ match for moduli of Higgs bundles is a key step in recent proofs of the $P=W$ conjecture \cite{dCHM, MS, HMMS, MSY}.

While the perverse filtration relies on the structure of a proper fibration, the Chern filtration is a more general structure that can be defined on many moduli spaces of sheaf-like objects. We believe that the Chern filtration is interesting to study on its own as a cohomological invariant. For example, one consequence of the $P=C$ match is the \textit{curious Hard Lefschetz symmetry} on the Chern filtration inherited from the relative Hard Lefschetz symmetry on the perverse filtration; see \cite{Mellit} and \cite[Remark 6.7 (iii)]{KLMP}. It is natural to ask if a similar symmetry holds for the Chern filtration where a nontrivial perverse filtration is not present, cf. Question \ref{question: symmetry}. In this paper, we present positive evidence to this question for $N_{r,d}$ in low ranks.

Finally, we note that the perverse counterpart of our new cohomological invariants have been systematically studied in the literature: the associated graded $\mathrm{gr}_\bullet^P H^*(-)$ of the perverse filtration on moduli of 1-dimensional sheaves are expected to be certain \textit{spaces of BPS states} in the Gopakumar--Vafa theory for (local) Calabi--Yau threefolds \cite{GV, HST, KL, MT}, and the refinement of Betti numbers by the perverse filtration gives \textit{refined BPS invariants} for such Calabi--Yau threefolds; see \cite[Section 6.2]{KLMP} for a detailed discussion and connections to other enumerative theories.

\subsection{Acknowledgements} We would like to thank Tamás Hausel, Andres Fernandez Herrero, Davesh Maulik, Dhruv Ranganathan, Junliang Shen, and Ravi Vakil for interesting and helpful discussions on relevant topics. We specially thank Anton Mellit for his lectures at Lisbon in 2022 which greatly helped us to construct the $\mathfrak{sl}_2$-triples. This research was supported by the Yonsei University Research Fund of 2024-22-0502.

\section{Preliminaries}

The purpose of this section is to collect some notations and results useful for the rest of the paper. In particular, we introduce the descendent algebra $\mathbb{D}$ and prove Theorem \ref{thm: topcherndegree}.

\subsection{Moduli stacks and moduli spaces}
\label{subsec: stacks}
Throughout, we fix a smooth projective curve $\Sigma$ of genus $g$ over the complex numbers. Let $r,d\in \BZ$ with $r>0$. We denote by $\FM_{r,d}$ the moduli stack of all vector bundles over $\Sigma$ of rank $r$ and degree $d$. For given $r>0$, the union of $\FM_{r,d}$ over all $d\in \BZ$ is the mapping stack 
\[\textup{Map}(\Sigma, B\textup{GL}_r)\,.\]
There is a universal bundle $\CV$ over the product $\FM_{r,d}\times \Sigma$. Recall that the slope of a bundle is defined as
\[\mu(V)=\frac{\deg(V)}{\rk(V)}\in \BQ\]
and that a bundle $V$ is said to be semistable if $\mu(W)\leq \mu(V)$ for every subbundle $W$; if the inequality is strict for any proper subbundle, then $V$ is said to be stable. We denote by $\FM_{r,d}^{\textup{ss}}\subseteq \FM_{r,d}$ the open moduli substack of semistable bundles. This stack admits a projective good moduli space $\FM_{r,d}^{\textup{ss}}\to M_{r,d}$. When $\gcd(r,d)=1$, every semistable bundle is stable and the good moduli map $\FM_{r,d}^{\textup{ss}}\to M_{r,d}$ is a trivial $\BG_m$-gerbe; in particular, $M_{r,d}$ is the rigidification of the stack $\FM_{r,d}^{\textup{ss}}$.

\begin{rmk}
In Section \ref{sec: intermediatestacks}, we will consider more generally the open substacks $\FM_{r,d}^{\leq \mu}$ wich parametrize bundles that do not admit subbundles with slope $>\mu$; this coincides with $\FM_{r,d}^{\textup{ss}}$ when we take $\mu=d/r$. 
\end{rmk}

We will mostly be interested in moduli spaces with fixed determinant. Let $\PPic^d=\FM_{1,d}=\FM_{1,d}^{\textup{ss}}$ be the Picard stack of $\Sigma$ and $\Pic^d=M_{1,d}$ be the Picard variety. Given a fixed line bundle $\Lambda$ of degree $d$, the stack of vector bundles with fixed determinant $\Lambda$ is the following fiber:
\begin{center}
    \begin{tikzcd}
\FN_{r,\Lambda}\arrow[r]\arrow[d]
& \FM_{r,d}\arrow[d, "\det"]\\
\{\Lambda\} \arrow[r]& {\PPic^d.}
    \end{tikzcd}
\end{center}
Since different choices of line bundles $\Lambda$ lead to isomorphic stacks, we will often just denote $\FN_{r,d}=\FN_{r,\Lambda}$. Similarly, we define $\FN_{r,d}^{\textup{ss}}$; the stack $\FN_{r,d}^{\textup{ss}}$ admits a projective good moduli space $N_{r,d}$. Indeed, the latter is the fiber of the determinant map $M_{r,d}\to \Pic^d$ over $[\Lambda]\in \Pic^d$. The good moduli map $\FN_{r,d}^{\textup{ss}}\to N_{r,d}$ is a $\mu_r$-gerbe when $\gcd(r,d)=1$. 


\subsection{Tautological classes and the descendent algebra}
\label{subsec: descendent algebra}

The universal bundle $\CV$ on $\FM_{r,d}\times \Sigma$ can be used to define tautological classes on the stack $\FM_{r,d}$. Recall from the introduction that we fix a basis $\{\mathbf{1}_\Sigma, e_1,\ldots, e_{2g}, \mathsf{pt}\}$ of $H^*(\Sigma)$, with the convention that $e_i\cdot e_{i+g} =\mathsf{pt}$ for $1\leq i \leq g$. Define classes 
\[\alpha_k\in H^{2k-2}(\FM_{r,d})\;,\;\psi_{k,i}\in H^{2k-1}(\FM_{r,d})\;,\;\beta_k\in H^{2k}(\FM_{r,d}) \]
as the Künneth components in
\begin{equation*}
c_k(\CV) = \alpha_k \otimes \mathsf{pt} + \sum_{i=1}^{2g} \psi_{k,i}\otimes e_i + \beta_k\otimes \mathbf{1}_\Sigma \in H^*(\FM_{r,d})\otimes H^*(\Sigma)\,.
\end{equation*}
for $1\leq k \leq r$ and $1\leq i\leq 2g$ (note that $\alpha_1=d\in H^0(\FM_{r,d})$).

The pullback of the universal bundle $\CV$ via the map $\FN_{r,d}\times \Sigma\to \FM_{r,d}\times \Sigma$ defines a universal bundle $\CU$ on $\FN_{r,d}\times \Sigma$ together with an isomorphism $\det(\CU)\xrightarrow{\sim} q^*\Lambda$, where $q\colon \FN_{r,d}\times \Sigma\to \Sigma$ is the projection. The Chern classes of $\CU$ are given by 
\begin{equation*}
c_k(\CU) = \alpha_k \otimes \mathsf{pt} + \sum_{i=1}^{2g} \psi_{k,i}\otimes e_i + \beta_k\otimes \mathbf{1}_\Sigma \in H^*(\FN_{r,d})\otimes H^*(\Sigma)\,.
\end{equation*}
for $k\geq 2$, where we omit the pullbacks of $\alpha_k,\, \psi_{k,i},\, \beta_k$ to $H^\ast(\FN_{r,d})$, and $c_1(\CU)=d(1\otimes \mathsf{pt})$. In particular, the classes $\psi_{1,i}$ and $\beta_1$ pull back to $0$.

\begin{defn}[$\textup{SL}_r$ descendent algebra] \label{def: descendents}
Let $\BD$ be the free super-commutative and unital $\BQ$-algebra generated by the formal symbols
\[\alpha_k,\, \beta_k,\,\psi_{k,i}, \quad 2\leq k\leq r,~ 1\leq i \leq 2g\,,\]
i.e.,
\[\BD=\BQ[\{\alpha_k, \beta_k\}_{2\leq k\leq r}]\otimes \Lambda_\BQ(\{\psi_{k, i}\}_{\substack{2\leq k\leq r,\,1\leq i\leq 2g}}).\]
\end{defn}

The descendent algebra admits two gradings: the usual \textit{cohomological grading}, which assigns
\[
\deg(\alpha_k) = 2k-2, \quad \deg(\beta_k) = 2k, \quad \deg(\psi_{k,i}) = 2k-1\,;
\]
and the \textit{Chern grading}, which we denote by $\deg^C(-)$ and assigns
\[
\deg^C(\alpha_k) = \deg^C(\beta_k) = \deg^C(\psi_{k,i}) =k\,.
\]
Note that we have the inequalities
\begin{equation}
\label{eq: deg_ineq}
\deg^C(D)\leq \deg(D)\leq 2\deg^C(D)
\end{equation}
for every $D\in \BD$. The first inequality is an equality if and only if $D$ is a power of $\alpha_2$, while the second is an equality if and only if $D$ is a product of $\beta_k$'s. 

There is a realization map $\BD\to H^\ast(\FN_{r,d})$ which sends the formal symbols $\alpha_k,\, \beta_k,\,\psi_{k,i}$ to the corresponding cohomology classes. By restricting to the semistable loci, we define a realization map $\BD\to H^\ast(\FN_{r,d}^{\textup{ss}})$ as well. 

\smallskip

The next proposition collects some facts about the cohomology rings of the moduli spaces and moduli stacks involved.

\begin{prop}\label{prop: cohomologiesofstacks}
We have the following relations between the (rational) cohomology of the different moduli stacks and moduli spaces: 
\begin{enumerate}
    \item The realization map $\BD\to H^\ast(\FN_{r,d})$ is a ring isomorphism. 
    \item The tensor product map $\PPic^0\times \FN_{r,d}\to \FM_{r,d}$ sending $(L, U)$ to $L\otimes U$ induces a ring isomorphism 
    \begin{align*}H^\ast(\FM_{r,d})&\simeq H^\ast(\FN_{r,d})\otimes H^\ast(\PPic^0)\\
    &\simeq \BQ[\{\alpha_k\}_{2\leq k\leq r}]\otimes \BQ[\{\beta_k\}_{1\leq k\leq r}] \otimes \Lambda_\BQ(\{\psi_{k, i}\}_{\substack{1\leq k\leq r,\,1\leq i\leq 2g}})\,.
    \end{align*}
    More generally, the same is true for $\PPic^0\times \FN_{r,d}^{\leq \mu}\to \FM_{r,d}^{\leq \mu}$ and $\PPic^0\times \FN_{r,d}^{\textup{ss}}\to \FM_{r,d}^{\textup{ss}}$.
    \item  Assume $\gcd(r,d)=1$. We have an isomorphism of stacks $\FM_{r,d}^{\textup{ss}}\simeq M_{r,d}\times B\BG_m$, hence
    \[H^\ast(\FM_{r,d}^{\textup{ss}})\simeq H^\ast(M_{r,d})\otimes H^\ast(B\BG_m)\simeq H^\ast(M_{r,d})[\beta_1]\,.\]
    \item Assume $\gcd(r,d)=1$. The good moduli map $\FN_{r,d}^{\textup{ss}}\to N_{r,d}$ induces a ring isomorphism 
    \[H^\ast(N_{r,d})\simeq H^\ast(\FN_{r,d}^{\textup{ss}})\,.\]
    \item  Assume $\gcd(r,d)=1$. The tensor product map $\Pic^d\times N_{r,d}\to M_{r,d}$ induces a ring isomorphism 
    \[H^\ast(M_{r,d})\simeq H^\ast(N_{r,d})\otimes H^\ast(\Pic^0)\simeq H^\ast(N_{r,d})\otimes \Lambda_\BQ(\psi_{1,1}, \ldots, \psi_{1, 2g})\]
\end{enumerate}
\end{prop}

The parts (iii), (iv) and (v) are standard facts. Although we believe that parts (i) and (ii) are also well known to experts, we could not find the exact statements in the literature. We defer their proof to Appendix \ref{appendix}. 

Thanks to Proposition \ref{prop: cohomologiesofstacks} (iv), it makes sense to define a realization map to $N_{r,d}$ as
\[\BD\to H^\ast(\FN_{r,d}^{\textup{ss}})\simeq H^\ast(N_{r,d})\,.\]

\begin{rmk}\label{rmk: taut class from End}
In the literature, it is frequent to define tautological generators for $N_{2,1}$ using the endomorphism bundle. Indeed, the endomorphism bundle $\End \CU$ on $\FN_{2,1}^{\textup{ss}}\times \Sigma$ descends uniquely to $N_{2,1}\times \Sigma$ and it is traditional (cf. \cite[Remark 2]{New}) to define classes $\alpha, \beta, \psi_{i}\in H^\ast(N_{2,1})$ via
\[c_2(\End \CU)=2\alpha\otimes \pt+4\sum_{i=1}^{2g} \psi_{i}\otimes e_i -\beta\otimes {\bf 1}_\Sigma\in H^\ast(N_{2,1}\times \Sigma)\,. \]
These classes compare to the ones previously defined in the following way:
\[
\alpha = 2\alpha_2, \quad \psi_i = \psi_{2,i}, \quad \beta = -4 \beta_2.
\]
\end{rmk}

\subsection{Chern filtration and refined Poincaré series}\label{subsec: chernandomega}
Given any of the possible moduli spaces (with coprime $r,d$) or moduli stacks 
\[X=\FM_{r,d}, \,\FN_{r,d},\: \FM_{r,d}^{\textup{ss}},\: \FN_{r,d}^{\textup{ss}},\:\FM_{r,d}^{\leq \mu},\: \FN_{r,d}^{\leq \mu}, \:M_{r,d}, \: N_{r,d}\]
previously introduced, there is a Chern filtration $C_\bullet H^\ast(X)$ on its cohomology parallel to Definition~\ref{def: chern}. We define the refined Poincaré series $\Omega(X, q,t)$ for any of these spaces and stacks as in Section \ref{subsec: conjecturesresults}:
\[\Omega(X, q,t)\coloneqq\sum_{i, j\geq 0} \dim \gr^C_i H^{i+j}(X)q^it^j\,.\]
Note that stacks typically have infinitely many non-zero cohomology groups, so $\Omega(X,q,t)$ is a formal power series in the variables $q, t$. The following proposition compares the Chern filtrations on different stacks; while the statement is very natural, its proof is slightly technical and we spell it out for completeness.

\begin{prop}\label{cor: comparing the Chern filtration}
    Under the isomorphism in Proposition \ref{prop: cohomologiesofstacks} (ii), we have 
    $$C_\bullet H^\ast(\FM_{r,d}^{\leq \mu})\simeq C_\bullet H^\ast(\FN_{r,d}^{\leq \mu})\otimes C_\bullet H^\ast(\PPic^0),
    $$
    meaning that the filtration on the left-hand side matches the tensor product filtration on the right-hand side. 
\end{prop}
\begin{proof}
    Under the decomposition $\PPic^0\simeq B\BG_m\times \Pic^0$, it is clear that we have 
    $$C_\bullet H^*(\PPic^0) \simeq C_\bullet H^*(\Pic^0)\otimes C_\bullet H^*(B\BG_m)
$$
where $C_\bullet H^*(\Pic^0)$ and $C_\bullet H^*(B\BG_m)$ are defined by the normalized Poincaré line bundle $\BL$ and the universal line bundle $\CL$, respectively. 
In fact, the Chern filtrations on $H^k(\Pic^0)$ and $H^{2k}(B\BG_m)$ jump only at the $k$'th step.

Since the Chern filtrations on the both sides of the isomorphism
$$\Phi:H^*(\FM_{r,d}^{\leq \mu})\xrightarrow{\sim} H^*(\FN_{r,d}^{\leq \mu})\otimes H^*(\Pic^0)\otimes H^*(B\BG_m)
$$
are multiplicative, it suffices to consider the generators. We note that, by Newton's identity, the Chern filtration can equivalently be defined as the smallest multiplicative filtration such that
\begin{equation}\label{eq: chern character taut}
    p_*(\ch_k\CV\cdot q^*\gamma)\in C_k H^*(\FM_{r,d}^{\leq \mu}),\quad k\geq 0,\ \gamma\in H^*(\Sigma)\,.
\end{equation}
The isomorphism $\Phi$ sends \eqref{eq: chern character taut} to
$$\sum_{k_1+k_2+k_3=k}\sum_{i\in I}
\Big(p_*(\ch_{k_1} (\CU)\cdot q^*\gamma^{L}_i)\Big)\otimes
\Big(p_*(\ch_{k_2} (\BL)\cdot q^*\gamma^{R}_i)\Big)\otimes
\ch_{k_3}(\CL)
$$
where $\Delta_*\gamma=\sum_{i\in I}\gamma^{L}_i\otimes \gamma^{R}_i$, which clearly lies in the $k$'th step of the tensor product filtration.

Now we prove the other inclusion. Note that under the isomorphism $\Phi$, we have
$$\Phi\Big(p_*(\ch_1(\CV)\cdot q^*\pt)\Big) = r\cdot 1\otimes 1\otimes z,\quad \Phi\Big(p_*(\ch_1(\CV)\cdot q^*e_i)\Big)=p_*(\ch_1(\BL)\cdot q^*e_i),
$$
where $z\coloneqq c_1(\mathcal{L})$ and we used $\ch_1(\CU)=d(1\otimes \pt)$. This proves that the generators of $H^*(B\BG_m)$ and $H^*(\Pic^0)$ lie in the correct step of the Chern filtration on the left hand side. So we are left to show that
$$\Phi^{-1}\Big(p_*(\ch_k(\CU)\otimes q^*\gamma)\Big)\in C_k H^*(\FM^{\leq \mu}_{r,d}),\quad k\geq 1,\ \gamma\in H^*(\Sigma).
$$
We prove this by induction on $(2-\deg(\gamma), k)\in\{0,1,2\}\times \BZ_{\geq 1}$ with respect to the lexicographical ordering. The base case holds trivially. Suppose that the statement holds for any $(2-\deg(\gamma'),k')<_{\textnormal{lex}}(2-\deg(\gamma),k)$. Note that 
$$\Delta_*\gamma=\gamma\otimes \pt +\sum_{i\in I'}\gamma^L_i\otimes \gamma^R_i
$$
where $\deg(\gamma_i^L)>\deg(\gamma)$. Therefore, the image of $p_*(\ch_k(\CV)\cdot q^*\gamma)\in C_k H^*(\FM_{r,d}^{\leq \mu})$ via $\Phi$ is
\begin{multline}\label{eq: expression}
    p_*(\ch_k(\CU)\cdot q^*\gamma)+
\sum_{k_1+k_2+k_3=k,\: k_1<k}
\Big(p_*(\ch_{k_1} (\CU)\cdot q^*\gamma)\Big)\otimes
\Big(p_*(\ch_{k_2} (\BL)\cdot q^*\pt)\Big)\otimes
\ch_{k_3}(\CL)\\+
\sum_{k_1+k_2+k_3=k}\sum_{i\in I'}
\Big(p_*(\ch_{k_1} (\CU)\cdot q^*\gamma^{L}_i)\Big)\otimes
\Big(p_*(\ch_{k_2} (\BL)\cdot q^*\gamma^{R}_i)\Big)\otimes
\ch_{k_3}(\CL).
\end{multline}
By the induction hypothesis and the previous discussion about $H^*(B\BG_m)$ and $H^*(\Pic^0)$, every term in \eqref{eq: expression} except for the main term $p_*(\ch_k(\CU)\cdot q^*\gamma)$ is mapped to $C_k H^*(\FM_{r,d}^{\leq \mu})$ via $\Phi^{-1}$. This proves that $\Phi^{-1}\big(p_*(\ch_k(\CU)\cdot q^*\gamma)\big)\in C_k H^*(\FM_{r,d}^{\leq \mu})$, hence completing the proof.
\end{proof}

As a consequence of Propositions \ref{prop: cohomologiesofstacks} {and \ref{cor: comparing the Chern filtration}}, we have the equalities
\begin{align}\Omega(\FN_{r,d},q,t)&=\prod_{k=2}^r \frac{(1+q^k t^{k-1})^{2g}}{(1-q^{k}t^{k-2})(1-q^{k}t^{k})}, \label{eq: omegaNN}\\
\Omega(\FM_{r,d},q,t)&=\frac{(1+q)^{2g}}{1-q t}\Omega(\FN_{r,d},q,t), \label{eq: omegaMM}\\
\Omega(\FM_{r,d}^{\textup{ss}},q,t)&=\frac{1}{1-qt}\Omega(M_{r,d},q,t)=\frac{(1+q)^{2g}}{1-qt}\Omega(N_{r,d},q,t)\,. \label{eq: omegaMN}
\end{align}

    Observe that $\Omega(\FN_{r,d},q,t)$ is 
    a rational function 
    with the symmetry 
    \[q^{(r+2)(r-1)(g-1)}t^{r(r-1)(g-1)} \cdot \Omega(\FN_{r,d},q^{-1}, t^{-1})=\Omega(\FN_{r,d}, q,t)\,.\]
    This is precisely the same symmetry that we speculate (cf. Question \ref{question: symmetry}) for the refined Poincaré polynomial $\Omega(N_{r,1},q,t)$ of the stable loci with $d = 1$.

    Moreover, the $t=-1$ specialization of the rational function $\Omega(\FN_{r,d}, q, t)$ is given by
    \[\Omega(\FN_{r,d},q,-1)=\prod_{k=2}^r (1-(-q)^k)^{2g-2}\,.\]
    We conjecture (cf. Conjecture \ref{conj: t=-1}) that this agrees with $\Omega(N_{r,d}, q, -1)$, the $t=-1$ specialization for the semistable loci. Indeed, the $r=2$ case and numerical evidence in higher ranks suggest that the difference
    \[\Omega(\FN_{r,d}, q,t)-\Omega(N_{r,d}, q, t)\]
is always divisible by $(1+t)^{2g}$. In Section \ref{sec: intermediatestacks} we will study this difference in the $r=2$ case by considering the stratification $\{\FN_{2,1}^{\leq \mu}\}$. The factor of $(1+t)^{2g}$ will appear in the differences between refined Poincaré series of consecutive stacks thanks to the (quite intricate) modified Mumford relations, cf. 
the equalities following \eqref{eq: differenceinequality} and its proof.

\subsection{Proof of Theorem \ref{thm: topcherndegree}}\label{sec: 1.1}\,

Our first general result towards the understanding of the Chern filtration on moduli spaces of bundles is Theorem \ref{thm: topcherndegree}. The theorem is equivalent to the following two facts: 
\begin{enumerate}
    \item[(i)] For $D\in \BD$ with $\deg^C(D)<(r+2)(r-1)(g-1)$, we have $\int_{N_{r,d}}D=0$.
    \item[(ii)] There is $D\in \BD$ with $\deg^C(D)=(r+2)(r-1)(g-1)$ such that $\int_{N_{r,d}}D\neq 0$. 
\end{enumerate}
Both (i) and (ii) are essentially shown in the proofs of Propositions 8 and 9 in \cite{EK-pontryagin}; for completeness, we briefly explain the argument here.

Let 
\[D=\prod_{k=2}^r\left(\alpha_k^{n_k}\beta_k^{m_k}\prod_{j=1}^{2g}\psi_{k, j}^{p_{k,j}} \right)\in \BD\,,\]
where $n_k, m_k\in \BZ_{\geq 0}$ and $p_{k,j}\in \{0,1\}$ be a class in top cohomological degree, i.e., $\deg(D)=2(r^2-1)(g-1)$. We observe that
\begin{align}\label{eq: cdegearlkirwan} 2\deg^C(D)&=\deg(D)+2\sum_{k=2}^r n_k+\sum_{k=2}^r\sum_{j=1}^{2g}p_{k, j}\\
&=2(r^2-1)(g-1)+2\sum_{k=2}^r n_k+\sum_{k=2}^r\sum_{j=1}^{2g}p_{k, j}\,.\nonumber
\end{align}

The Jeffrey--Kirwan formula \cite{jeffreykirwan} expresses the integral
\[G(\epsilon)\coloneqq \int_{N_{r,d}} \exp\left(\sum_{k=2}^r \epsilon \delta_k \alpha_k\right) \prod_{k=2}^r\left(\beta_k^{m_k}\prod_{j=1}^{2g}\psi_{k, j}^{p_{k,j}} \right)\]
as a complicated iterated residue; here, $\delta_2, \ldots, \delta_k$ are formal variables and we regard $G(\epsilon)$ as a polynomial in $\epsilon$ with coefficients being polynomials in $\delta_2, \ldots, \delta_k$. Using this formula, it is shown in the proof of \cite[Proposition 8]{EK-pontryagin} that $G(\epsilon)$ is a polynomial divisible by $\epsilon^M$ where
\[M=\frac{1}{2}\left(2(r-1)(g-1)-\sum_{k=2}^r\sum_{j=1}^{2g}p_{k, j}\right)\,.\]
If $\int_{N_{r,d}}D\neq 0$, we would get a contribution to $\epsilon^{n_2+\cdots+n_k}$, hence
\[\sum_{k=2}^r n_k\geq M\,.\]
Using \eqref{eq: cdegearlkirwan}, this inequality is equivalent to $\deg^C(D)\geq (r+2)(r-1)(g-1)$, showing (i).

\smallskip

For (ii), Earl and Kirwan show \cite[Proposition 9]{EK-pontryagin} that $\int_{N_{r,d}}D\neq 0$ for
\[D=\eta \cdot \alpha_2^{(r-1)(g-1)}\,,\]
where $\eta$ is a polynomial in the classes $\beta_2, \ldots, \beta_r$ with cohomological degree $2r(r-1)(g-1)$. Explicitly, $\eta$ is given by writing the symmetric polynomial in formal variables $X_1, \ldots, X_r$
\[\eta=\prod_{i<j} (X_i-X_j)^{2(g-1)}\]
in terms of elementary symmetric polynomials, and setting the $k$-th elementary symmetric polynomial to $\beta_k$ (where $\beta_1$ is interpreted to be 0). Since the Chern degree of a polynomial in $\beta$ classes is half of its cohomological degree, we have
\[\deg^C(D)=r(r-1)(g-1)+2(r-1)(g-1)=(r+2)(r-1)(g-1)\,.\eqno\qed\]

The following corollary is a generalization of the known fact that the subring of $H^\ast(N_{r,d})$ generated by $\beta$ classes vanishes beyond cohomological degree $2r(r-1)(g-1)$, cf. \cite[Theorem 7]{EK-pontryagin}.

\begin{cor} 
Let $r(r-1)(g-1)\leq m\leq 2(r^2-1)(g-1)$. Then $C_k H^m(N_{r,d})=0$ for 
\[k<m-r(r-1)(g-1)\,.\]
\end{cor}
\begin{proof}
Suppose that $D\in C_k H^m(N_{r,d})$ is non-zero. By Poincaré duality, there is 
\[E\in H^{2(r^2-1)(g-1)-m}(N_{r,d})\] 
such that $\int_{N_{r,d}}D\cdot E\neq 0$. Recall that the inequality \eqref{eq: deg_ineq} states that the Chern degree is bounded above by the cohomological degree, so 
\[E\in C_{2(r^2-1)(g-1)-m}H^{2(r^2-1)(g-1)-m}(N_{r,d})\,\]
and hence
\[D\cdot E\in C_{2(r^2-1)(g-1)-m+k}H^{2(r^2-1)(g-1)}(N_{r,d})\,.\]
By Theorem \ref{thm: topcherndegree} it follows that 
\[2(r^2-1)(g-1)-m+k\geq (r+2)(r-1)(g-1)\]
which is equivalent to
\[k\geq m-r(r-1)(g-1)\,.\qedhere\]
\end{proof}

Theorem \ref{thm: topcherndegree} together with \eqref{eq: omegaMN} gives a similar result for the moduli space $M_{r,d}$ of bundles without fixed determinant:
\begin{cor}
The top Chern degree of $M_{r,d}$ is $r(r+1)(g-1)+2$.
\end{cor}

The proof of Theorem \ref{thm: topcherndegree} also shows the following:

\begin{prop}\label{prop: dindependence}
    If $D\in \BD$ is such that $\deg^C(D)=(r+2)(r-1)(g-1)$, then $\int_{N_{r,d}}D$ does not depend on the degree $d$.
\end{prop}
\begin{proof}
    Recall from the proof of Theorem \ref{thm: topcherndegree} that the integrals of $D$ with top Chern degree are encoded in the $\epsilon^M$ coefficient of $G(\epsilon)$. The only part dependent on $d$ in the Jeffrey--Kirwan \cite{jeffreykirwan} formula for $G(\epsilon)$ is $\tilde c$, which appears only in the term
    \[\exp\!\big(\epsilon \mathrm{d}q_X(w\tilde c)\big)\,.\]
    But this term does not contribute to the coefficient of $\epsilon^M$, so $\int_{N_{r,d}}D$ is independent of $d$ for $D$ of top Chern degree.
\end{proof}

\subsection{The associated graded and its (non-degenerate) pairing}\label{subsec: associatedgraded}

We explain in this subsection the relation between the numerical symmetry of a filtration and the perfectness of certain pairings on its associated graded; this is discussed briefly after Question \ref{question: symmetry}. We also explain the obstruction for the symmetry \eqref{eq: main} to hold for \textit{all} degrees $d$ coprime to a given rank.

\begin{lem}\label{lem: sym implies perfect}
    Let $V$ be a finite-dimensional vector space with an increasing filtration $C_\bullet V$ and a perfect pairing $\langle -,-\rangle: V\otimes V \to \mathbb{Q}$. Assume that there exists $N>0$ such that
    \begin{enumerate}
    
        \item[(i)] The pairing $C_i V \otimes C_j V \to \mathbb{Q}$ is zero for $i+j < N$.

        \item[(ii)] The numerical symmetry $\dim \mathrm{gr}^C_i V = \dim \mathrm{gr}^C_{N-i} V$ holds.
    \end{enumerate}
    Then there exists a well-defined {perfect} pairing
    \begin{equation}\label{eq: gr_pairing}
    \langle -, - \rangle^\mathrm{gr}:\mathrm{gr}^C_\bullet V \otimes \mathrm{gr}^C_{N-\bullet}V \longrightarrow \mathbb{Q}.
    \end{equation}
\end{lem}

\begin{proof}
    By (i), it is clear that the pairing on $V$ descends to a well-defined pairing \eqref{eq: gr_pairing}. We define subspaces
    \[
    \widetilde{C}_i V \coloneqq \{w\in V\mid \langle v,w \rangle=0, ~ \textrm{for all } v\in C_{N-1-i}V\}.
    \]
    By (i), we have $C_i V \subset \widetilde{C}_i V$. By (ii) and the perfectness of the pairing on $V$, we have
    \begin{align*}
        \dim \widetilde{C}_i V &= \dim V - \dim C_{N-1-i} V = \sum_{j\leq i} \dim \mathrm{gr}^C_{N-j} V = \sum_{j\leq i} \dim \mathrm{gr}^C_{j} V = \dim C_i V.
    \end{align*}
    It follows that $C_i V = \widetilde{C}_i V$. If $\overline{w}\in \mathrm{gr}^C_{N-i} V$ is such that $\langle \overline{v}, \overline{w}\rangle^\mathrm{gr}=0$ for all $\overline{v} \in \mathrm{gr}^C_i V$, then by definition 
    \[
    w \in \widetilde{C}_{N-1-i} V = C_{N-1-i} V,
    \]
    and thus $\overline{w}=0$. This shows the perfectness of \eqref{eq: gr_pairing}.
\end{proof}

For our purpose, take $C_\bullet V$ to be $H^*(N_{r,d})$ equipped with the Chern filtration (Definition~\ref{def: chern}),  $\langle -, -\rangle$ to be the Poincaré pairing, and $N=(r+2)(r-1)(g-1)$ to be the top Chern degree (Theorem \ref{thm: topcherndegree}). The following corollary is then immediate from Lemma \ref{lem: sym implies perfect} and the fact that $\langle v, w\rangle \neq 0$ only if $v$ and $w$ have complementary cohomological degrees.

\begin{cor} \label{cor: 1.4}
    Assume that the numerical Chern symmetry holds:
    \[
    \overline{\Omega}(N_{r,d}, q, 1) = \overline{\Omega}(N_{r,d}, q^{-1}, 1).
    \]
    Then the refined symmetry $\overline{\Omega}(N_{r,d}, q,t) = \overline{\Omega}(N_{r,d}, q^{-1}, t^{-1})$ also holds, and the pairing 
    \begin{equation}\label{eq: grH_pairing}
    \langle -,-\rangle^\mathrm{gr}:\mathrm{gr}^C_\bullet H^*(N_{r,d}) \otimes \mathrm{gr}^C_{N-\bullet}H^*(N_{r,d}) \to \mathbb{Q}
    \end{equation}
    is perfect.
\end{cor}

When the symmetry does hold (for example for $N_{2,1}$), the perfectness of the pairing \eqref{eq: grH_pairing} implies that the associated graded ring $\mathrm{gr}^C_\bullet H^*(N_{r,d})$ is determined by integrals \textit{only in the top Chern degree}, instead of all integrals, which are necessary to determine the full cohomology ring $H^*(N_{r,d})$. To state this precisely, we introduce the ``graded integral'' notation
\[\int_{N_{r,d}}^{\gr}D\]
which is defined to be $\int_{N_{r,d}}D$ if $D\in \BD$ has the top Chern degree, and zero otherwise. We equip the descendent algebra with a symmetric pairing $\langle -,-\rangle^\gr:\BD\otimes \BD\rightarrow\BQ$ given by
$$\langle D,D'\rangle^\gr \coloneqq \int_{N_{r,d}}^{\gr} D\cdot D'\,.$$
Then the non-degeneracy of the pairing on $\mathrm{gr}^C_\bullet H^*(N_{r,d})$ means that the ideal
\[I^\gr= \ker\big(\BD\twoheadrightarrow \gr^C_\bullet H^\ast(N_{r,d})\big)\]
is precisely the kernel of the pairing $\langle -,-\rangle^\gr:\BD\otimes \BD\rightarrow\BQ$.

In general, we expect the integrals in the top Chern degree to be ``better behaved'' than all the integrals. The $r=2$ case illustrates this phenomenon well. A formula for integrals of general descendents can be found in \cite[(30)]{Tha1}, whose most complicated aspect is the appearance of the Bernoulli numbers $B_q$. However, considering only integrals in the top Chern degree is the same as setting $q=0$, which eliminates this complexity. Indeed, the integrals with top Chern degree in rank two are completely determined by the ``Virasoro proportionality'' \eqref{eq: Virasoro} and monodromy invariance. Another evidence for this philosophy can be seen in Proposition~\ref{prop: dindependence}, which states that integrals in the top Chern degree, unlike general integrals, do not depend on the degree $d$.

\begin{rmk}\label{rmk: nosymmetry}
    The above discussions implies that the Chern symmetry \eqref{eq: main} cannot hold for every degree $d$: if the symmetry holds for $N_{r,d}$, then $\mathrm{gr}^C_\bullet H^*(N_{r,d})$, and in particular the Betti numbers of $N_{r,d}$, are determined by the graded integral functional $\int_{N_{r,d}}^{\gr}$. But this functional does not depend on $d$, by Proposition \ref{prop: dindependence}, so the symmetry holding for every $d$ would imply that the Betti numbers of $N_{r,d}$ were the same for every $d$, which is not true (for example $N_{5,1}$ and $N_{5,2}$ have different Betti numbers). Numerical evidence suggests that the Betti numbers of $N_{r,\pm 1}$ are smaller than the Betti numbers of $N_{r,d}$ with $d\not\equiv \pm 1 \mod r$, which provides some support to Question \ref{question: symmetry}. This situation contrasts with the case of moduli spaces $M_{\beta, \chi}^H(S)$ of 1-dimensional sheaves  on del Pezzo surfaces $S$, where we do expect the analogous symmetry for every $\chi$ (indeed, it is a consequence of the $P=C$ conjecture, as discussed in \cite{KLMP}) and the Betti numbers are known to be independent of $\chi$ by \cite{MSchiind}; we expect that a result analogous to Proposition \ref{prop: dindependence} holds in that setting and hope to address this in the future.
\end{rmk}

\section{The refined Poincaré polynomial of \texorpdfstring{$N_{2,1}$}{N21}}
\label{sec:refinedpoincareN21}

In this section we will prove the following formula for the refined Poincaré polynomial of the moduli space of rank 2 vector bundles.

\begin{thm}[Theorem \ref{thm: main}]\label{thm: omegaM21}
    The refined Poincaré polynomial for $N_{2,1}$ is
    \begin{equation*}
    \Omega(N_{2,1},q,t) = \frac{(1+q^2t)^{2g}-q^{2g}(1+t)^{2g}}{(1-q^2)(1-q^2t^2)}.
    \end{equation*}
    In particular, the symmetry \eqref{eq: main} holds when $r=2$.
\end{thm}

Note that setting $q=t$ recovers the known Poincaré polynomial for $N_{2,1}$ \cite{newsteadtop, HN, AB}.

\smallskip

This formula will be obtained by using the explicit basis found by Zagier in \cite{Zag_rank2}. Before we recall the construction and properties of this basis, we prove the following general lemma that explains how the refined Poincaré polynomial of $N_{r,d}$ can be calculated provided we have a basis with certain desirable properties. 

\begin{lem}\label{lem: antidiagonalbasis}
Let $N=N_{r,d}$ with $\gcd(r,d)=1$. Suppose that there exists a basis $\CB$ of $H^\ast(N)$ that has the following properties:
\begin{enumerate}
    \item[(i)] The intersection pairing matrix is ``anti-diagonal'', i.e., there is an involution $v\mapsto \check v$ of $\mathcal B$ such that $\int_N v\cdot u=0$ for $u\in \mathcal{B}\setminus \{\check v\}$ and $\int_N v\cdot \check v \neq 0$. 
    
    \item[(ii)] There are non-negative integers $\ell_v, k_v$ for each $v\in \CB$ such that $v\in C_{\ell_v}H^{k_v}(N)$. In particular,
    \[k_v+k_{\check v}=2(r^2-1)(g-1)\,.\]
    \item[(iii)] We have
    \[\ell_v+\ell_{\check v}=(r+2)(r-1)(g-1)\,.\]
\end{enumerate}
Then
\[\dim \mathrm{gr}^C_{\ell}H^k(N)=\#\{v\in \CB\colon k_v=k\,,\,\ell_v=\ell\}\,.\]
In particular, the symmetry \eqref{eq: main} holds. 
\end{lem}
\begin{proof}
We begin by proving that $C_\ell H^\ast(N)$ is spanned by $v\in \CB$ with $\ell_v\leq \ell$. Let $w\in C_\ell H^\ast(N)$ and consider its decomposition into the basis $\CB$, 
\[w=\sum_{v\in \CB}a_vv\,,\]
for some $a_v\in \BQ$. If $\ell_v>\ell$ then $\ell_{\check v}+\ell<\ell_{\check v}+\ell_v=(r+2)(r-1)(g-1)$, so by Theorem \ref{thm: topcherndegree} we know that $\int_N w\cdot \check v=0$. But by the anti-diagonal property of the basis it implies that $a_v=0$, so $w$ is a linear combination of $v\in \CB$ with $\ell_v\leq \ell$. 

It follows that the (projections of the) basis vectors $v\in \CB$ with $\ell_v=\ell$ span $\mathrm{gr}^C_{\ell}H^\ast(N)$. To prove the claim of the lemma, it is enough to show that these are linearly independent in $\mathrm{gr}^C_{\ell}H^\ast(N)$. Suppose not, and there exists a linear combination
\[w=\sum_{v\in \CB_\ell}a_v v\in C_{\ell-1}H^\ast(N)\,\]
where the sum is over $\CB_\ell=\{v\in \CB\colon \ell_v=\ell\}$. Let $u\in \CB$ be any element of the basis. If $\check u\notin \CB_\ell$ then $\int_{N}u\cdot w=0$ by the anti-diagonal property. On the other hand, if $\ell_{\check u}=\ell$ then $\ell_u+\ell-1=(r+2)(r-1)(g-1)-1$, so we also conclude that $\int_N u\cdot w=0$ by Theorem \ref{thm: topcherndegree}. By Poincaré duality it follows that $w=0$, so all $a_v$ are equal to 0.

We conclude the symmetry \eqref{eq: main} since the involution defines a bijection between the sets
\[\{v\in \CB\colon k_v=k\,,~\ell_v=\ell\}\]
and
\[\{v\in \CB\colon k_v=2(r^2-1)(g-1)-k\,,~\ell_v=(r+2)(r-1)(g-1)-\ell\}\,.\qedhere\]
\end{proof}

We shall now recall the construction in \cite[Section 4]{Zag_rank2} of a basis with the properties above. Denote
\[
\gamma = -2\sum_{i=1}^g \psi_{i}\psi_{i+g}.
\]
We define the classes 
\[\xi_{r,s} \coloneqq \sum_{l=0}^{\min(r,s)} \binom{r+s-l}{r}\beta^{s-l}\frac{(2\gamma)^l}{l!}\xi_{r-l}\in \BQ[\alpha, \beta, \gamma]\,,\]
where $\xi_r\in \mathbb{Q}[\alpha, \beta, \gamma]$ are certain classes of cohomological degree $2r$ that can be defined recursively; see Theorem 3 in loc. cit. (and note that $\xi_{r,s}=\xi_{r,s,0}$). The classes $\xi_r$ are closely related to the ``Mumford relations'' that we shall introduce in Section \ref{sec: intermediatestacks}, but their precise shapes will not be relevant for now; in the notation of Section \ref{sec: intermediatestacks}, $\xi_r=c_{1,r}$ are the coefficients of \eqref{eq: phid}. Since the Chern degree is bounded by the cohomological degree, we have $\xi_r\in C_{2r}H^{2r}(N)$, and hence it is easy to see that
\[\xi_{r,s}\in C_{2r+2s}H^{2r+4s}(N)\,.\]
Let $[g] \coloneqq \{1,2,\ldots, g\}$. Given subsets $A, B, C\subseteq [g]$, we denote 
\[\psi_A=\prod_{j\in A}\psi_j\,,\;\psi^\ast_B=\prod_{j\in B}\psi_{j+g}\,,\;\gamma_C=\prod_{j\in C}\psi_j\psi_{j+g}\,.\]
Note that the elements
\begin{equation}\label{eq: basiselements}\xi_{r,s}\psi_A\psi^\ast_B\gamma_C\,\end{equation}
with $r,s\geq 0$ and $A, B, C\subseteq [g]$ pairwise disjoint form a basis (cf. \cite[Theorem~6]{Zag_rank2}) 
of the descendent algebra
\[\BD=\BQ[\alpha, \beta]\otimes \Lambda(\psi_1, \ldots, \psi_{2g})\,.\]
Zagier uses these classes to construct a basis of $H^\ast(N)$ with the necessary properties for us. His basis $\CB$ admits a partition into ``blocks'' 
\begin{equation}\label{eq: basispartition}\CB=\bigsqcup \CB_{r,s,A,B, c}\end{equation}
where the union runs over all $r,s,A,B, c$ with $r,s,c\geq 0$, $A,B\subseteq [g]$ disjoint, with the restriction that 
\[c'\coloneqq g-1-r-s-|A|-|B|-c\geq 0\,.\]

Moreover, the number of elements in the block $\CB_{r,s,A,B, c}$ is equal to
\[|\CB_{r,s,A,B, c}|=\binom{g-|A|-|B|}{\min(c,c')}\,.\]

When $c\leq c'$, the elements of $\CB_{r,s,A,B, c}$ are precisely \eqref{eq: basiselements} with $C$ running over all subsets of $[g]\setminus (A\cup B)$ with $c$ elements. When $c>c'$, the elements of $\CB_{r,s,A,B, c}$ are still linear combinations of such classes. In particular, it follows that if $v\in \CB_{r,s,A,B, c}$ then $v\in C_{\ell_v}H^{k_v}(N)$ with
\[\ell_v=2r+2s+2|A|+2|B|+4c\,\textup{ and }\,k_v=2r+4s+3|A|+3|B|+6c\,.\]
This basis has the anti-diagonal property (condition (i) in Lemma \ref{lem: antidiagonalbasis}) with respect to the Poincaré pairing; moreover, if $v\in \CB_{r,s,A,B, c}$ then the unique $\check v\in \CB$ that pairs non-trivially with $v$ is in the block
\[\check v\in \CB_{s,r,B,A,c'}\,.\]
In particular, a straightforward calculation shows that
\[\ell_v+\ell_{\check v}=4g-4\,.\]

\begin{proof}[Proof of Theorem \ref{thm: omegaM21}]
By the discussions above, Zagier's basis $\mathcal{B}$ satisfies the conditions of Lemma~\ref{lem: antidiagonalbasis}, so $\Omega(N_{2,1}, q,t)$ can be calculated via the combinatorics of the basis. Indeed, we obtain
\[\Omega(N_{2,1}, q,t)=\sum \binom{g-|A|-|B|}{\min(c,c')}q^{2r+2s+2|A|+2|B|+4c} \cdot t^{2s+|A|+|B|+2c}\,\]
where the sum is over the same set as \eqref{eq: basispartition}. This sum can be evaluated to
\[\frac{(1+2q^2 t+q^4 t^2)^{2g}-(q^2+2q^2 t+q^2)^{2g}}{(1-q^2)(1-q^2 t^2)}= \frac{(1+q^2t)^{2g}-q^{2g}(1+t)^{2g}}{(1-q^2)(1-q^2t^2)}\]
by following the computations in \cite[p. 550]{Zag_rank2} while keeping track of the Chern degrees. We sketch the key steps and indicate the necessary modifications for readers' convenience. As in loc. cit., we set $p=\min(c, c')$ and $l=|c-c'|$; then we have
\[
\Omega(N_{2,1},q,t) = \sum_{\substack{p,l,r,s,h\geq 0\\2p+l+r+s+h=g-1}}2^h \binom{g}{h}\binom{g-h}{p}q^{2r+2s+2h+4p}t^{2s+h+2p}(1+q^{4l}t^{2l}-\delta_{0,l}).
\]
We have the following refinement of an identity appearing in \cite{Zag_rank2}:
\[
\sum_{\substack{l,r,s\geq 0\\l+r+s=n-1}}q^{2r+2s}t^{2s}(1+q^{4l}t^{2l}-\delta_{0,l})=\frac{(1-q^{2n})(1-q^{2n}t^{2n})}{(1-q^2)(1-q^2t^2)},
\]
and thus we find
\begin{align*}
    \Omega(N_{2,1},q,t) &= \sum_{\substack{p,n,h\geq 0\\2p+n+h=g}} \binom{g}{h}\binom{g-h}{p}(2q^2t)^h \frac{(q^{4}t^{2})^p-(q^{4}t^{2})^p q^{2n}-(q^{4}t^{2})^p (qt)^{2n}+(q^{4}t^{2})^{n+p}}{(1-q^2)(1-q^2t^2)}\\
    &=\sum_{\substack{h,i,j\geq 0\\h+i+j=g}}\frac{g!}{h!\cdot i!\cdot j!} (2q^2t)^h \frac{(q^4t^2)^j-(q^2)^i(q^2t^2)^j}{(1-q^2)(1-q^2t^2)}\\ &=\frac{(1+2q^2t+q^4t^2)^{g}-(q^2+2q^2t+q^2t^2)^{g}}{(1-q^2)(1-q^2t^2)}.
\end{align*}
In the last equality we used the trinomial theorem (as Zagier did). \qedhere

\end{proof}

\section{Chern filtration of intermediate stacks}
\label{sec: intermediatestacks}

The goal of this section is to prove Theorem \ref{thm: intermediatestacks}. Recall that $\FM^{\leq d}_{2,1}$ is the stack of vector bundles on $\Sigma$ of rank 2 and degree 1 that do not have destabilizing line bundles of degree $>d$. In particular, $\FM_{2,1}^{\leq 0}=\FM_{2,1}^{\textup{ss}}$ is the stack of stable vector bundles. Similarly, given a line bundle $\Lambda$ of degree 1, we have a stack $\FN^{\leq d}_{2,1}=\FN^{\leq d}_{2,\Lambda}$ parametrizing vector bundles with fixed determinant $\Lambda$, which is defined as a fiber of the determinant map to the Picard stack:
\begin{center}
    \begin{tikzcd}
\FN_{2,1}^{\leq d}\arrow[r]\arrow[d]
& \FM_{2,1}^{\leq d}\arrow[d, "\det"]\\
\{\Lambda\} \arrow[r]& {\PPic^d.}
    \end{tikzcd}
\end{center}

All the stacks $\FM^{\leq d}_{2,1}$ are finite type open substacks of $\FM_{2,1}$, so we have a chain of inclusions
\[\FM_{2,1}^{\textup{ss}}=\FM_{2,1}^{\leq 0}\subseteq \FM_{2,1}^{\leq 1}\subseteq  \ldots\subseteq  \FM_{2,1}^{\leq d-1}\subseteq  \FM_{2,1}^{\leq d}\subseteq \ldots\subseteq \FM_{2,1}\,.\]
The stacks $\FM_{2,1}^{\leq d}$ exhaust $\FM_{2,1}$ in the sense that their union is $\FM_{2,1}$ and the complements $\FM_{2,1}\setminus \FM_{2,1}^{\leq d}$ have increasingly large codimension when $d\rightarrow \infty$. Hence, 
\[\lim_{d\to \infty} H^\ast(\FM_{2,1}^{\leq d})=  H^\ast(\FM_{2,1})\]
in the sense that, for fixed cohomological degree $k$, the finite dimensional cohomology groups $H^k(\FM_{2,1}^{\leq d})$ stabilize to $H^k(\FM_{2,1})$. The basic strategy used in \cite{HN, AB} to calculate the Betti numbers of $\FM_{2,1}^{\textup{ss}}$ (or more generally $\FM^{\textup{ss}}_{r, d}$) is to use the knowledge of the cohomology of $\FM_{2,1}$, see Proposition \ref{prop: cohomologiesofstacks} (ii), and this stratification. This strategy also gives the Betti numbers of the intermediate stacks $\FM_{2,1}^{\leq d}$. We review this idea here to setup the notation for the proof of Theorem \ref{thm: intermediatestacks}.\footnote{The same strategy can be applied directly to $\FN_{2,1}^{\leq d}$, but it is slightly more natural on $\FM_{2,1}^{\leq d}$, so we opt to explain it in that setting and extract consequences regarding $\FN_{2,1}^{\leq d}$ by using Proposition \ref{prop: cohomologiesofstacks} (ii).}

Let $j\colon  \FM_{2,1}^{\leq d-1}\hookrightarrow  \FM_{2,1}^{\leq d}$ be the open embedding and let $\mathcal Z_d$ be the complement\footnote{In the language used in Appendix \ref{appendix}, $\CZ_d=\FM_{2,1}^{=(1,d)}$.}
\[\mathcal Z_d=\FM_{2,1}^{\leq d}\setminus \FM_{2,1}^{\leq d-1}\,.\]
This is the stack parametrizing vector bundles whose maximal destabilizing line bundle has degree exactly $d$, so any such vector bundle sits in a maximally destabilizing short exact sequence
\[0\to L_d\to V\to L_{1-d}\to 0\]
such that $L_d, L_{1-d}$ are line bundles of degrees $d$ and $1-d$, respectively. Hence, $\CZ_d$ is the the total space of the vector bundle stack over $\FM_{1,d}\times \FM_{1, 1-d}$ whose fiber over $(L_d, L_{1-d})$ is $[\Ext^1(L_{1-d}, L_d)/\Hom(L_{1-d},L_d)]$. The inclusion $\iota\colon \CZ_d\hookrightarrow \FM_{2,1}^{\leq d}$ is a closed embedding of codimension $2d+g-2$.

It is shown in \cite{H} that there is a short exact sequence
\begin{equation}\label{eq: sesrecursion}
0\rightarrow H^{*-2(2d+g-2)}(\CZ_d)\xrightarrow{\iota_\ast} H^\ast(\FM_{2,1}^{\leq d})\xrightarrow{j^\ast} H^\ast(\FM_{2,1}^{\leq d-1})\rightarrow 0\,.
\end{equation}
This implies a relation between the Betti numbers of $\FM_{2,1}^{\leq d}$ and $\FM_{2,1}^{\leq d-1}$, which we can write in terms of the $q=t$ specialization 
\begin{align*}\Omega(\FM_{2,1}^{\leq d}, t,t)&=\Omega(\FM_{2,1}^{\leq d-1}, t,t)+t^{2g+4d-4}\Omega(\FM_{1,d}, t,t)\cdot \Omega(\FM_{1,1-d}, t,t)\\
&=\Omega(\FM_{2,1}^{\leq d-1}, t,t)+t^{2g+4d-4}\frac{(1+t)^{4g}}{(1-t^2)^2}.
\end{align*}
Hence,
\begin{align*}\Omega(\FM_{2,1}^{\leq d}, t,t)&=\Omega(\FM_{2,1}^{\textup{ss}}, t,t)+\sum_{k=1}^d t^{2g+4k-4}\frac{(1+t)^{4g}}{(1-t^2)^2}\\
&=\frac{(1+t)^{2g}}{1-t^2}\Omega(N_{2,1}, t,t)+\sum_{k=1}^d t^{2g+4k-4}\frac{(1+t)^{4g}}{(1-t^2)^2}
\end{align*}
where the last line uses \eqref{eq: omegaMN}. When $d\rightarrow \infty$ the left hand side converges to $\Omega(\FM_{2,1}, t,t)$, which is known by Proposition \ref{prop: cohomologiesofstacks}(ii). We can deduce from the equality above that
\begin{align*}\Omega(\FM_{2,1}^{\leq d}, t,t)&=\frac{(1+t^3)^{2g}-t^{2g+4d}(1+t)^{2g}}{(1-t^2)(1-t^4)}\,.
\end{align*}
Theorem \ref{thm: intermediatestacks} is a refinement of this formula.

\subsection{Strategy of the proof
}\label{subsec: strategyintermediatestacks}

The statement of Theorem \ref{thm: intermediatestacks} is equivalent to the equality
\begin{equation}\label{eq: desireddifference}\Omega(\FN^{\leq d}_{2,1}, q,t)-\Omega(\FN^{\leq d-1}_{2,1}, q,t)=q^{2g+4d-4}\frac{(1+t)^{2g}(1+q^2)}{1-q^2 t^2}\,.
\end{equation}
We denote by $j_{\Lambda}$ the open inclusion $\FN_{2,1}^{\leq d-1}\hookrightarrow \FN_{2,1}^{\leq d}$. We show in Lemma \ref{lem: split Gysin} that $j^\ast_{\Lambda}$ is also surjective.

Our strategy will be to identify a basis of $\ker(j_{\Lambda}^\ast)$ (see Proposition \ref{prop: basismr3}) and bound the Chern degree of the elements of this basis (see Proposition \ref{prop: boundCmodifiedmr}). To show that this bound is tight,  we use our knowledge of $\Omega(\FN_{2,1},q,t)$ (cf. \eqref{eq: omegaNN}) and $\Omega(N_{2,1},q,t)$ (cf. Theorem \ref{thm: main}). The following definition will help us in making this argument precise: 

\begin{defn}
Let $\Omega_1, \Omega_2$ be generating series in $q,t$ with integer coefficients: 
\[\Omega_s(q,t)=\sum_{i,j\geq 0}a_{ij}^s q^i t^j, \quad s=1,2\,.\]
We say that $\Omega_1\preceq \Omega_2$ if for every $0\leq \ell\leq k$ we have
\[\sum_{\substack{i+j=k\\i\leq \ell}}a_{ij}^1\leq \sum_{\substack{i+j=k\\i\leq \ell}}a_{ij}^2\]
with equality when $k=\ell$. 
\end{defn}

Note that $\preceq$ makes $(\BZ\llbracket q,t\rrbracket, +)$ a partially ordered abelian group. This strange looking definition is motivated by the following observation: if $H^\ast$ is a graded vector space admitting two different filtrations $C^1_\bullet H^\ast\subseteq C^2_\bullet H^\ast$, then $\Omega_1\preceq \Omega_2$ where
\[\Omega_s(q,t)=\sum_{i,j\geq 0}\dim \gr^{C^s}_i H^{i+j}q^it^j,\quad s=1,2\,.\]
Suppose now that we are able to show the inequality\footnote{Since the Chern filtrations on $H^\ast(\FN^{\leq d}_{2,1})$ and $H^\ast(\FN^{\leq d-1}_{2,1})$ are defined to be the filtrations induced by the surjections $\BD\twoheadrightarrow H^\ast(\FN^{\leq d}_{2,1})\twoheadrightarrow H^\ast(\FN^{\leq d-1}_{2,1})$, the first equality is a consequence of \cite[Tag 05SP]{stacks}.}
\begin{align}\label{eq: differenceinequality}\Omega(\FN^{\leq d}_{2,1}, q,t)-\Omega(\FN^{\leq d-1}_{2,1}, q,t)&=\sum_{i,j\geq 0}\dim \gr_i^C\big(\!\ker(j^\ast_{\Lambda})\cap H^{i+j}(\FN^{\leq d}_{2,1})\big)q^it^j \\
&\preceq q^{2g+4d-4}\frac{(1+t)^{2g}(1+q^2)}{1-q^2 t^2}\,.\nonumber
\end{align}

Summing these inequalities from $d=1$ to $d=\infty$ we obtain
\begin{align*}
q^{2g}\frac{(1+t)^{2g}}{(1-q^2)(1-q^2 t^2)}&=\Omega(\FN_{2,1}, q,t)-\Omega(N_{2,1}, q,t)\\
&=\Omega(\FN_{2,1}, q,t)-\Omega(\FN_{2,1}^{\leq 0}, q,t)\\
&\preceq \sum_{d=1}^\infty q^{2g+4d-4}\frac{(1+t)^{2g}(1+q^2)}{(1-q^2t^2)}\\
&=q^{2g}\frac{(1+t)^{2g}(1+q^2)}{(1-q^4)(1-q^2t^2)}\,.
\end{align*}

Since the first and last terms are equal, all the inequalities \eqref{eq: differenceinequality} must actually be equalities, which would prove Theorem \ref{thm: intermediatestacks}. The remaining of this section is dedicated to the proof of~\eqref{eq: differenceinequality}.

\subsection{A basis of Mumford relations}

Mumford proposed a set of relations (see \cite[Section~9]{AB}) among the generators of $H^\ast(N_{2,1})$ or, equivalently, among the generators of $H^\ast(\FM_{2,1}^{\textup{ss}})$, which was proven by Kirwan to be complete \cite{Kir} (and in higher ranks by \cite{EK}). We review these relations now and explain their connection to the short exact sequence \eqref{eq: sesrecursion}.

Fix a point $\pt\in \Sigma$. Let $\BL=\BL_{1-d}$ be the Poincaré line bundle on $M_{1, 1-d}\times \Sigma$ normalized at $\pt$, i.e. its restriction to $M_{1,1-d} \times \pt$ being trivial. Recall that $M_{1, 1-d}$ is the Picard variety of degree $1-d$ line bundles; if there is no possibility for confusion, we abbreviate the notation and write $M_{1,1-d}=\Pic$. 
Consider the complex
\[\BK=\BK_{d}=R\CH om_p(\CV, \BL)\]
on $\FM_{2,1}\times \Pic$, where $\CV$ is (the pullback of) the universal bundle on $\FM_{2,1}\times \Sigma$ and \[p\colon \FM_{2,1}\times \Pic\times \Sigma\to \FM_{2,1}\times \Pic\] is the projection. We will also consider the restriction of $\BK$ to the open substacks $\FM_{2,1}^{\leq d'}\times \Pic$ and still denote it in the same way. Let $p_1, p_2$ be the projections of $\FM_{2,1}\times \Pic$ onto the first and second factors, respectively. To avoid confusion with the $\psi$ classes on $\FM_{2,1}$, we will denote by $\epsilon_i=\psi_{1,i}\in H^\ast(\Pic)$. Given a subset $I\subseteq [2g]$ we write $\epsilon_I=\prod_{i\in I}\epsilon_i$. 

\begin{defn}
Given $k\geq 0$ and
\[A\in H^\ast(\Pic)=\Lambda_\BQ(\epsilon_1, \ldots, \epsilon_{2g})=\mathrm{span}\{\epsilon_{I}\colon I\subseteq [2g]\}\,,\] 
we let\footnote{These are technically the {``generalized''} Mumford relations in the sense of \cite{EK}, but they are equivalent to the original Mumford relations (see the proof of Proposition \ref{prop: explicitmr}). 
}
\[\mathsf{MR}_{k, A}^d=(p_1)_\ast\big(c_k(-\BK)\cdot p_2^\ast A\big)\in H^\ast(\FM_{2,1})\,.\]
\end{defn}

We denote in the same way the restrictions of $\mathsf{MR}_{k, A}^d$ to $H^\ast(\FM_{2,1}^{\leq d'})$ for any $d'$. A fundamental observation is that the restriction of $\BK[1]$ to $\FM_{2,1}^{\leq d-1}$ is a vector bundle. Indeed, if $V$ corresponds to a point in the stack $\FM_{2,1}^{\leq d-1}$ and $L$ is a vector bundle of degree $1-d$, then $\Hom(V, L)=0$. So the restriction of $\mathcal{H}^0(\BK)$ to $\FM_{2,1}^{\leq d-1}\times \Pic$ vanishes and $\BK[1]=\mathcal{H}^1(\BK)$ is a vector bundle of rank
\[\dim \Ext^1(V, L)=-\chi(V, L)=2d+2g-3\,.\]
Hence, we find that 
\[\mathsf{MR}^d_{k, A}\in \ker(j^\ast)\subseteq H^\ast(\FM_{2,1}^{\leq d})\]
for $k\geq 2d+2g-2$. We will now use the short exact sequence \eqref{eq: sesrecursion} to identify a basis of $\ker(j^\ast)$ in terms of these relations.

\begin{prop}\label{prop: basismr1}
Let $d\geq 1$. The subspace $\ker(j^\ast)$ has a basis given by 
\[\Big\{\psi_{1, I}\beta_1^\ell \mathsf{MR}_{k+2g+2d-2, \epsilon_J}^d\colon \ell, k\geq 0, ~ I, J\subseteq [2g]\Big\}\,.\]
\end{prop}

\begin{proof}
The inclusion $\iota\colon \CZ_d\hookrightarrow \FM_{2,1}^{\leq d}$ can be factored as 
\begin{center}
\begin{tikzcd}
\CZ_d\arrow[r, "\pi"]
&\FM_{2,1}^{\leq d}\times \Pic\arrow[r, "p_1"]& \FM_{2,1}^{\leq d}\,.
\end{tikzcd}
\end{center}

The map $\pi$ is a virtual projective bundle in the sense of \cite[Definition 4.1]{Hyeonjun Park_virtual pullback}; see also \cite{Qingyuan Jiang_Derived projectivizations of complexes} which uses derived geometry: 
\[\pi\colon \CZ_d=\BP(\BK)\to \FM_{2,1}^{\leq d}\times \Pic.\]
{Denote by $\CO(1)$ the tautological bundle for $\pi$ and set $z \coloneqq c_1(\CO(1))$. }By the definition of $\BK$ and adjunction, over $\BP(\BK)\times \Sigma$ there is a morphism of vector bundles $\pi^\ast \CV\to \pi^\ast \BL(1)$; the universal maximally destabilizing sequence on $\CZ_d\times \Sigma$ is
\[0\to \CL_{d}\to \pi^\ast \CV\to \pi^\ast \BL(1)=\CL_{1-d}\to 0\]
and in particular we have\[\CL_d+\CL_{1-d}=\pi^\ast \CV~\textup{ and }~\CL_{1-d}=\pi^\ast \mathbb{L}(1)\,\]
in the $K$-theory of $\CZ_d\times \Sigma$. The line bundles $\CL_d$ and $\CL_{1-d}$ are the pullbacks of the universal line bundles on $\FM_{1,d}$ and $\FM_{1,1-d}$, respectively. Recall that 
\[H^\ast(\CZ_d)\simeq H^\ast(\FM_{1,d})\otimes H^\ast(\FM_{1,1-d})\simeq H^\ast(\FM_{1,d})\otimes H^\ast(\Pic)[z]\,.\]
So we can conclude easily that 
\[\Big\{\pi^\ast\big(\psi_{1,I}\beta_1^\ell\otimes \epsilon_J\big)z^k\colon \ell, k\geq 0,~ I, J\subseteq [2g] \Big\}\]
is a basis of $H^\ast(\CZ_d)$. 

By the short exact sequence \eqref{eq: sesrecursion}, we obtain a basis of $\ker(j^\ast)$ by applying $\iota_\ast=(p_1)_\ast\circ \pi_\ast$ to the basis above. We have 
\begin{align*}
\iota_\ast\Big(\pi^\ast\big(\psi_{1,I}\beta_1^\ell\otimes \epsilon_J\big)z^k\Big)&=(p_1)_\ast\Big(\big(\psi_{1,I}\beta_1^\ell\otimes \epsilon_J\big)\pi_\ast(z^k)\Big)\\
&=\psi_{1, I}\beta_1^\ell (p_1)_\ast\Big(c_{k-\mathrm{rk}(\BK)+1}(-\BK)p_2^\ast \epsilon_J\Big)\\
&=\psi_{1, I}\beta_1^\ell \mathsf{MR}_{k+2g+2d-2, \epsilon_J}^d\,,
\end{align*}
where the second equality uses the pushforward formula along virtual projective bundles (which can be shown as in the proof of \cite[Proposition 4.2]{Hyeonjun Park_virtual pullback}). This finishes the proof.\qedhere
\end{proof}

\begin{rmk}
    Note that the proposition above reproves (and refines, in the sense that it identifies a precise basis of relations) the completeness of the Mumford relations in rank 2, originally shown in \cite{Kir}.
\end{rmk}

We now turn to the stacks with fixed determinant. Recall Proposition \ref{prop: cohomologiesofstacks} which, in particular, says that the morphism $H^\ast(\FM_{2,1})\to H^\ast(\FN_{2,1})$, induced by the natural map $\FN_{2,1}\to \FM_{2,1}$, is the quotient by $\psi_{1,i}=0$ and $\beta_1=0$. Recall also that $j_{\Lambda}$ denotes the open inclusion $\FN_{2,1}^{\leq d-1}\hookrightarrow \FN_{2,1}^{\leq d}$ and that $j_{\Lambda}^\ast$ is surjective. By Proposition \ref{prop: basismr1} we can conclude the following corollary:

\begin{cor}\label{cor: basismr2}
The subspace $\ker(j^\ast_{\Lambda})$ has a basis given by
\[\Big\{\mathsf{MR}_{k+2g+2d-2, \epsilon_J}^d\colon k\geq 0, ~J\subseteq [2g]\Big\}\,.\]
\end{cor}

Note that here we still denote by $\mr^d_{k, A}$ the pullback of the Mumford relations via the pullback map $H^\ast(\FM_{2,1}^{\leq d})\to H^\ast(\FN_{2,1}^{\leq d})$, which sets $\beta_1=\psi_{1,i}=0$.

\subsection{Modified Mumford relations and a Chern degree bound}

A priori, given a general class $A\in H^\ast(\Pic)$, there is no obvious way to improve the trivial bound on the Chern degree of the Mumford relations 
\[\mathsf{MR}^d_{k+2g+2d-2, A}\in C_{2k+2g+4d-4+\deg(A)}H^{2k+2g+4d-4+\deg(A)}(\FN_{2,1})\,.\]
However, to prove inequality \eqref{eq: differenceinequality} we need to construct relations with lower Chern degrees. It turns out that the bound above can be improved to
\[\mathsf{MR}^d_{k+2g+2d-2, A}\in C_{2k+2g+4d-4}H^{2k+2g+4d-4+\deg(A)}(\FN_{2,1})\,.\]
if $A$ is a primitive class in the sense of \eqref{eq: prim}; however, for $A$ non-primitive, a similar improvement is not possible. Instead, we will take a linear combination of Mumford relations for which we can show a better bound and, in particular, we will exhibit $2^{2g}$ relations with Chern degree $2g+4d-4$. This will require a detailed study of the explicit formula for Mumford relations obtained in \cite{kiem, Zag_rank2}. 

\smallskip

Let 
\[\theta=2\sum_{i=1}^g \epsilon_i\epsilon_{i+g}\in H^2(\Pic)\,.\]
Note that $\theta$ is twice the standard theta divisor of the abelian variety $\Pic$; in particular it is ample. The cohomology of the Picard variety admits a Lefschetz decomposition 
\[H^\ast(\Pic)=\bigoplus_{l=0}^g \bigoplus_{m=0}^{g-l}\theta^m\cdot \mathsf{Prim}_l\]
where 
\begin{equation}
\label{eq: prim}
\mathsf{Prim}_l=\ker\Big(\theta^{g-l+1}\colon H^l(\Pic)\to H^{2g-l+2}(\Pic) \Big)\subseteq H^l(\Pic)\,.
\end{equation}
For instance, $\epsilon_I\in \mathsf{Prim}_l$ if $I\subseteq [2g]$ has $l$ elements and the property that there is no $i$ such that $i,\, i+g\in I$. Given $\tilde \sigma_l\in \mathsf{Prim}_l$ of the form $\sum_{I}a_I \epsilon_I$, we will consider the corresponding class 
\[\sigma_l=\sum_I a_I \psi_I\in H^{3l}(\FN_{2,1})\,.\]

We introduce the generating series
\begin{equation}
    \label{eq: phid}
\Phi_d(t)=\sum_{n=0}^\infty c_{d,n} t^n=(1-\beta t^2)^{d-\frac{3}{2}}e^{-\frac{2\gamma t}{\beta}}\Big(\frac{1+t\sqrt{\beta}}{1-t\sqrt{\beta}}\Big)^{\frac{\alpha}{2\sqrt{\beta}}+\frac{\gamma}{\beta\sqrt{\beta}}}\,.\end{equation}
The coefficients $c_{d,n}$ are elements of $\BD^\Gamma=\BQ[\alpha, \beta, \gamma]/(\gamma^{g+1})$ of cohomological degree $2n$. Note that when $d=1$ this generating series coincides with $\Phi(t)$ from \cite{kiem}; it also appears in \cite{Zag_rank2} where we have $c_{1,r} = \xi_r$ (we use a new notation to avoid confusion with Zagier's basis $\xi_{r,s}$ in Section \ref{sec:refinedpoincareN21}). 

In the next proposition we will write down an explicit formula for the Mumford relations. Recall the falling factorial notation:
\[(x)_m\coloneqq \prod_{j=0}^{m-1}(x-j)\,.\]
We will use two basic properties of the falling factorial repeatedly: the symmetry
\[(-x)_m=(-1)^m(x+m-1)_m\]
and the falling factorial binomial theorem
\[(x+y)_m=\sum_{j=0}^m \binom{m}{j}(x)_j(y)_{m-j}\,.\]

\begin{prop}\label{prop: explicitmr}
Let $k, m\geq 0$, $d\geq 1$, and $\tilde \sigma_l\in \mathsf{Prim}_l$. Then\footnote{We use the notation $[t^n]$ for the operator extracting the $t^n$ coefficient of a generating series in the formal variable $t$.}
\[\mr^d_{k, \theta^m \tilde\sigma_l}=(-1)^l2^{2g-m-k}[t^{k+m-g-l}]\left(\Phi_d(t)\cdot \sum_{j=0}^m \binom{m}{j}(g-l-j)_{m-j}(1-\beta t^2)^{m-j}(-2 \gamma t^3)^{j}\right)\sigma_l.\]
\end{prop}
\begin{proof}

Consider the morphism
\[\phi\colon \FN_{2,1}\times \Pic\to \FM_{2, 4g+2d-5}\]
which sends $(U, L)\mapsto U\otimes L^\vee\otimes \omega_\Sigma$. Let $\mathcal V$ be the universal bundle on $\FM_{2, 4g+2d-5}\times \Sigma$. By the definition of the morphism,
\[(\phi\times \textup{id})^\ast \CV=\CU\otimes \BL^\vee\otimes \omega_\Sigma=R\CH om(\BL, \CU\otimes \omega_\Sigma)\,.\]
Note that we are abusing notation slightly by writing $\BL, \CU, \omega_\Sigma$ for the pullbacks to $\FN_{2,1}\times \Pic\times \Sigma$ via the obvious projections. Let us denote both the projections $\FM_{2, 4g+2d-5}\times \Sigma\to \FM_{2, 4g+2d-5}^{\textup{rig}}$ and $\FN_{2,1}\times \Pic\times \Sigma\to \FN_{2,1}\times \Pic$ by $p$. 
By Grothendieck--Verdier duality, 
\[\phi^\ast\big(Rp_\ast \mathcal V\big)=Rp_\ast(\phi\times \textup{id})^\ast \CV=Rp_\ast R\CH om(\BL, \CU\otimes \omega_\Sigma)=\Big(Rp_\ast R\CH om(\CU, \BL)\Big)^\vee [-1] =\big(\BK[1]\big)^\vee\,.\]
Hence, 
the Chern classes of $\BK$ are pullbacks of the Chern classes of $Rp_\ast \mathcal V$. Concretely,
\[c_{2t}(-\BK)=\phi^\ast c_{-2t}(Rp_\ast\CV)\,.\]
Now the latter has been computed by Zagier \cite[(29)]{Zag_rank2} in the $d=1$ case. By a simple modification of his calculation, we find that 
\begin{equation}\label{eq: c-2t}
c_{-2t}(Rp_\ast\CV)=\Phi_d(t)G(t)
\end{equation}
holds modulo\footnote{The equality modulo $\beta_1$ is enough for our purposes since $\phi^\ast \beta_1=0$.} $\beta_1$, where
\[G(t)=(1-\beta t^2)^g\exp\left(\frac{2\theta t+2B t^2-2\gamma t^3}{1-\beta t^2}\right)\]
and
\[B=2\sum_{i=1}^g (\epsilon_i\psi_{i+g}-\epsilon_{i+g}\psi_{i})\,.\]
Note that
\[(\phi\times \textup{id})^\ast c_1(\CV)=1\otimes (4g+2d-5)\pt-2\sum_{i=1}^{2g}\epsilon_i\otimes e_i\,,\]
so the pullback $\phi^\ast$ sends the classes $\psi_{1,i}\in H^\ast(\FM_{2, 4g+2d-5})$ (denoted by $d_i$ in \cite{kiem, Zag_rank2}) to $-2\epsilon_i$; in particular, it sends the class $A$ in loc. cit. to $2\theta$. 

It is shown in \cite[Page 4]{kiem}\footnote{There is a slight difference in signs and factors of 2 between our formula and the formulas in loc. cit. due to the fact that they are extracting coefficients with respect to $\psi_{1,i}$ classes, while we do so with $\epsilon_i$ classes, which corresponds to integrating over $\Pic$.} that
\[\int_{\Pic} G(t)\theta^m \tilde \sigma_l=2^{2g-m}(-1)^l t^{g+l-m}\sigma_l(1-\beta t^2)^m e^{-\frac{2\gamma t^3}{1-\beta t^2}}\sum_{p\geq 0}(g-l-p)_m\frac{(2\gamma t^3)^p}{p!(1-\beta t^2)^p}\,.\]
To shorten the notation, set $x=g-l$. Kiem simplifies the formula above when $m=0,1,2$, which are equations $(0), (1), (2)$ in \cite{kiem}, respectively. More generally, we can derive a formula for arbitrary $m$. First, note that
\begin{align*}(x-p)_m&=(-1)^m(p-x+m-1)_m=(-1)^m\sum_{j=0}^m \binom{m}{j}(-x+m-1)_{m-j}(p)_j\\
&= \sum_{j=0}^m (-1)^j\binom{m}{j}(x-j)_{m-j}(p)_j\,.\end{align*}
Therefore, 
\begin{align*}\sum_{p\geq 0}(x-p)_m\frac{(2\gamma t^3)^p}{p!(1-\beta t^2)^p}&=\sum_{p\geq 0}\sum_{j=0}^m (-1)^j\binom{m}{j}(x-j)_{m-j}(p)_j\frac{(2\gamma t^3)^p}{p!(1-\beta t^2)^p}\\
&=\sum_{j=0}^m \binom{m}{j}(x-j)_{m-j}\frac{(-2\gamma t^3)^j}{(1-\beta t^2)^j}\sum_{p\geq j}\frac{(2\gamma t^3)^{p-j}}{(p-j)!(1-\beta t^2)^{p-j}}\\
&=e^{\frac{2\gamma t^3}{1-\beta t^2}} \sum_{j=0}^m \binom{m}{j}(x-j)_{m-j}\frac{(-2\gamma t^3)^j}{(1-\beta t^2)^j}\,.
\end{align*}
Hence, we conclude that
\[\int_{\Pic} G(t)\theta^m \widetilde \sigma_l=(-1)^l 2^{2g-m} t^{g+l-m}\sigma_l\sum_{j=0}^m\binom{m}{j}(x-j)_{m-j}(1-\beta t^2)^{m-j} (-2\gamma t^3)^j\,.\]
Together with \eqref{eq: c-2t}, this concludes the proof.\qedhere
\end{proof}

Suppose now that we take $m=0$ in the proposition, which gives
\[\mr_{k, \tilde \sigma_l}^d=(-1)^l 2^{2g-k} c_{d, k-g-l}\sigma_l\,.\]
Since we trivially have
\begin{equation}
    \label{eq: trivialCbound}c_{d,n}\in H^{2n}(\FN_{2,1})=C_{2n}H^{2n}(\FN_{2,1})~\textup{ and }~\sigma_l\in C_{2l}H^{3l}(\FN_{2,1})\end{equation}
we conclude that 
\[\mr_{k, \tilde \sigma_l}^d\in C_{2k-2g}H^{2k-2g+l}(\FN_{2,1})\,.\]
However, when $m>0$ a similar bound is not true. To illustrate this, consider the case $m=1$, in which Proposition \ref{prop: explicitmr} gives
\[\mr_{k, \theta\cdot \tilde \sigma_l}^d=(-1)^l 2^{2g-k-1}\sigma_l\big((g-l)c_{d, k-g-l+1}-(g-l)\beta c_{d, k-g-l-1}-2\gamma c_{d, k-g-l-2}\big)\,.\]
The second and third terms of the right hand side have Chern degree bounded by $2k-2g$, but not the first term. The key observation is that the first term is proportional to the Mumford relation 
$\mr_{k+1, \tilde \sigma_l}^d$; hence, we find that 
\[\widetilde{\mr}_{k, \theta\cdot \tilde\sigma_l}^d\coloneqq -\mr_{k, \theta\cdot \tilde \sigma_l}^d+(g-l)\mr_{k+1, \tilde \sigma_l}^d\]
also has Chern degree bounded by $2k-2g$. More generally, for arbitrary $m$, we define modified Mumford relations as follows:
\begin{defn}
Let $k, m\geq 0$, $d\geq 1$, and $\tilde \sigma_l\in \mathsf{Prim}_l$. We define the modified Mumford relations $\widetilde{\mr}_{k, \theta^m\cdot \tilde\sigma_l}^d$ as
\[\widetilde{\mr}_{k, \theta^m\cdot \tilde \sigma_l}^d=\sum_{s=0}^m (-1)^{s}\binom{m}{s}(g-l-s)_{m-s}\mr^d_{k+m-s, \theta^s\cdot \tilde \sigma_l}\in H^\ast(\FN_{2,1})\,.\]
\end{defn}
Note that Corollary \ref{cor: basismr2} is equivalent to the statement that
\[\widetilde \mr^d_{k+2g+2d-2, A}\]
form a basis of $\ker(j^\ast_{\Lambda})$, when $k$ runs through the non-negative numbers and $A$ runs through a basis of $H^\ast(\Pic)$; this is the case since the matrix relating the original Mumford relations and the modified Mumford relations is triangular with $\pm 1$ entries along the diagonal. The following proposition establishes the desired bound on the Chern degree.

\begin{prop}\label{prop: boundCmodifiedmr}
Let $A=\theta^m\cdot \tilde \sigma_l\in H^\ast(\Pic)$. Then
\begin{equation}
\label{eq: modified MR}
\widetilde{\mr}_{k, A}^{d}=(-1)^l 2^{2g-m-k} \frac{m!}{(g-l-m)!} \sum_{\substack{b+c=m\\a+b+2c=k-g-l}}(g-l-c)!c_{d,a}\frac{\beta^b}{b!}\frac{(2\gamma)^c}{c!}\sigma_l\,.
\end{equation}
In particular, we have
\[\widetilde{\mathsf{MR}}_{k, A}^d\in C_{2k-2g}H^{2k-2g+|A|}(\FN_{2,1})\,.\]
\end{prop}
\begin{proof}
    Denote $x=g-l$. By Proposition \ref{prop: explicitmr}, the modified Mumford relation $\widetilde{\mathsf{MR}}_{k, \theta^m \cdot \tilde \sigma_l}^d$ is given by
    \[(-1)^l2^{2g-m-k}[t^{k+m-g-l}]\left(\Phi_d(t)\cdot F(t) \right)\sigma_l\]
where
    \[F(t)=\sum_{s=0}^m (-1)^{s}\binom{m}{s}(x-s)_{m-s}\sum_{j=0}^s \binom{s}{j}(x-j)_{s-j}(1-\beta t^2)^{s-j}(-2\gamma t^3)^{j}\,.\]
    Now observe that
    \[(x-j)_{s-j}(x-s)_{m-s}=(x-j)_{m-j}\]
    and
    \[\binom{m}{s}\binom{s}{j}=\frac{m!}{(m-s)!j!(s-j)!}=\binom{m}{j}\binom{m-j}{s-j}\,.\]
    Hence, we can rewrite $F$ as
    \begin{align*}F(t)&=\sum_{j=0}^m \binom{m}{j}(x-j)_{m-j}(2\gamma t^3)^j\sum_{s=j}^m \binom{m-j}{s-j}(\beta t^2-1)^{s-j}\\
    &=\sum_{j=0}^m \binom{m}{j}(x-j)_{m-j}(2\gamma t^3)^j(\beta t^2)^{m-j}\,.
    \end{align*}
    This shows formula \eqref{eq: modified MR} after setting $a=k-m-k-l-j$, $b=m-j$ and $c=j$. Using the trivial bound \eqref{eq: trivialCbound}, we conclude that if $a,b,c$ are as in the sum  \eqref{eq: modified MR}, the $c_{d,a}\beta^{b}\gamma^c\sigma_l$ has Chern degree bounded by
\[2a+2b+4c+2l=2k-2g\,,\]
which proves the last part of the proposition.\qedhere
\end{proof}


\subsection{Concluding the proof}

The basis from Corollary \ref{cor: basismr2} together with the bound from Proposition \ref{prop: boundCmodifiedmr} are still not enough to obtain the inequality \eqref{eq: differenceinequality}. Although Proposition \ref{prop: boundCmodifiedmr} is an optimal bound for the Chern degree in the formal algebra $H^\ast(\FN_{2,1})\simeq \BD$, it is no longer optimal in $H^\ast(\FN_{2,1}^{\leq d})$. The next lemma provides a mechanism to further improve the Chern degree bound.

\begin{lem}\label{lem: mrdifferentd}
For any $d\in \BZ$, $k\geq 0$ and $A\in H^\ast(\Pic)$, we have the identity
\[\mr_{k, A}^{d+1}=\mr_{k, A}^d-\beta_1\mr_{k-1, A}^d+\beta_2\mr_{k-2, A}^d\]
in $H^\ast(\FM_{2,1})$.
\end{lem}
\begin{proof}
    We can identify the Jacobians $M_{1, -d}$ and $M_{1, 1-d}$, which parametrize line bundles with different degrees, via the isomorphism sending $L\mapsto \CO_\Sigma(\pt)\otimes L$; denote both of these moduli spaces by $\Pic$. Under this identification, the universal line bundle $\BL_{-d}$ is the same as $\BL_{1-d}\otimes \CO_\Sigma(-\pt)$. Therefore we obtain a short exact sequence of sheaves on $\Pic\times \Sigma$
\[0\rightarrow \BL_{1, -d}\rightarrow \BL_{1, 1-d}\rightarrow (\BL_{1,1-d})_{|\Pic\times \pt}=\CO_{\Pic}\otimes \CO_{\pt}\rightarrow 0\,.\]

Applying $R\mathcal{H}om_p(\CV, -)$ to this exact sequence of sheaves and rotating, we obtain an exact triangle
\[p_1^\ast \CV^\vee_{|\FM_{2,1}\times \pt}\to \BK_{d+1}[1]\to \BK_d[1]\,\]
in $\FM_{2,1}\times \Pic$. Note that $\CV^\vee_{|\FM_{2,1}\times \pt}$ is a vector bundle of rank 2 with first and second Chern classes $-\beta_1$ and $\beta_2$, respectively. The result follows from the exact triangle above and the definiton of the Mumford relations.
\end{proof}

\begin{prop}\label{prop: basismr3}
Let $d\geq 1$. The subspace $\ker(j^\ast_{\Lambda})$ has a basis formed by the relations
\[\beta^{\ell} \widetilde{\mathsf{MR}}^d_{k+2g+2d-2, A}\in H^\ast(\FN_{2,1}^{\leq d})\,\]
where $\ell\geq 0$, $k\in \{0,1\}$, and $A$ runs through a basis of $H^\ast(\Pic)$.
\end{prop}
\begin{proof}
Recall that the restriction of $\mr_{k, A}^{d+1}$ to $\FM_{2,1}^{\leq d}$ vanishes for $k\geq 2d+2g$. Therefore, pulling back the relation of Lemma \ref{lem: mrdifferentd} along $\FN_{2,1}^{\leq d}\to \FM_{2,1}^{\leq d}\to \FM_{2,1}$, we obtain the equality
\[\mr_{k, A}^d=-\beta_2 \mr_{k-2, A}^d=\frac{1}{4}\beta \cdot \mr_{k-2, A}^d\]
in $H^\ast(\FN_{2,1}^{\leq d})$ for all $k\geq 2d+2g$.

From the definition of modified Mumford relations, it is clear that we also have 
\[\widetilde \mr_{k, A}^d=\frac{1}{4}\beta \cdot \widetilde \mr_{k-2, A}^d\,.\]

By Corollary \ref{cor: basismr2} and the fact that the matrix relating $\mr$ and $\widetilde \mr$ is invertible, the relations $\widetilde{\mathsf{MR}}^d_{k+2g+2d-2, A}$ form a basis of $\ker(j^\ast_{\Lambda})$ when $k$ runs through $\BZ_{\geq 0}$ and $A$ runs through a basis of $H^\ast(\Pic)$. Together with the identity above, the conclusion follows. \qedhere
\end{proof}

The basis from the last proposition and the bound on the Chern degree are finally enough to prove \eqref{eq: differenceinequality}. Indeed, we have
\[\beta^{\ell} \widetilde{\mathsf{MR}}^d_{k+2g+2d-2, A}\in C_{2k+2g+4d-4+2\ell} H^{2k+2g+4d-4+4\ell+|A|}(\FN_{2,1}^{\leq d})\,.\]
Therefore, 
\begin{align}\label{eq: inequalitiesomegad}
\Omega(\FN^{\leq d}_{2,1}, q, t)-\Omega(\FN^{\leq d-1}_{2,1}, q, t)&\preceq \sum_{\substack{\ell\geq 0\\
k\in \{0,1\},\, A}}q^{2k+2g+4d-4+2\ell}t^{2\ell+|A|}\\
&=q^{2g+4d-4}\frac{(1+t)^{2g}(1+q^2)}{1-q^2t^2}\nonumber
\end{align}
As explained in Section \ref{subsec: strategyintermediatestacks}, this inequality is enough to conclude the proof of Theorem \ref{thm: intermediatestacks}.

\subsection{A basis for $I^{\gr}_{\leq d}$} \label{subsec: relationsIgr}
We finish this section by providing a very explicit description of the ideals
\[I^{\gr}_{\leq d}\coloneqq \ker\big(\BD\to \gr^C_\bullet H^\ast(\FN_{2,1}^{\leq d})\big)\,.\]
By Corollary \ref{cor: basismr2}, the ideal $I_{\leq d}\coloneqq \ker\big(\BD\to H^\ast(\FN_{2,1}^{\leq d})\big)$ admits a basis formed by 
\[\mr_{k, A}^{d'}\in H^\ast(\FN_{2,1})\simeq \BD\] 
where $d'> d,\; k\geq 2g+2d'-2$, and $A$ runs through a basis of $H^\ast(\Pic)$. By Lemma \ref{lem: mrdifferentd}, we can replace this basis by  
\[\beta^\ell \widetilde{\mr}_{k, A}^{d+1}\in H^\ast(\FN_{2,1})\simeq \BD\] 
where $\ell\geq 0$ and $k\geq 2g+2d$. Moreover, the analysis in this section shows that the projections of $\beta^\ell \widetilde{\mr}_{k, A}^{d+1}$ in $\gr^C_{2\ell+2k-2g} \BD$ form a basis of $I^{\gr}_{\leq d}=\gr^C_\bullet I_{\leq d}$.

\smallskip

Recall the formula from Proposition \ref{prop: boundCmodifiedmr}: 
\[\widetilde{\mr}_{k, \theta^m\cdot \tilde \sigma_l}^{d+1}=(-1)^l 2^{2g-m-k} \frac{m!}{(g-l-m)!} \sum_{\substack{b+c=m\\a+b+2c=k-g-l}}(g-l-c)!c_{d+1,a}\frac{\beta^b}{b!}\frac{(2\gamma)^c}{c!}\sigma_l\,.\]
Note that $c_{d+1,a}$ is equal to $\alpha^a/a!$ modulo terms of Chern degrees $\leq{2a-1}$, since $\Phi_d(t)$ is equal to $e^{t\alpha}$ modulo the ideal generated by $\beta, \gamma$. Thus, the projection of $\widetilde{\mr}_{k, \theta^m\cdot \tilde \sigma_l}^{d+1}$ in $\gr^C_{2k-2g} \BD$ gives, up to constant factors, the relation 
\begin{equation}\label{eq: relationsIgr}
    \sum_{\substack{b+c=m\\a+b+2c=k-g-l}}(g-l-c)!\frac{\alpha^a}{a!}\frac{\beta^b}{b!}\frac{(2\gamma)^c}{c!}\sigma_l\in I^{\gr}_{\leq d}\,
\end{equation}
for any $l,m,k,\sigma_l$ such that $l+m\leq g, \; k\geq 2g+2d$, and $\sigma_l\in \mathsf{Prim}_l$. The preceding analysis shows that $I^{\gr}_{\leq d}$ is freely generated as a $\BQ[\beta]$-module by the relations \eqref{eq: relationsIgr}.

\section{The \texorpdfstring{$\mathfrak{sl}_2$}{sl2}-triples 
}
\label{sec: sl2}

In this final section we prove Theorem \ref{thm: sl2}. This is done by first constructing $\mathfrak{sl}_2$-triples formally over the descendent algebra $\mathbb{D}$, and then showing that they descend via the surjections $\mathbb{D}\to \mathrm{gr}^C_\bullet H^*(\FN_{2,1}^{\leq d})$. We give two proof of the descending property. The first applies only to the case $d=0$ and shows something stronger, namely that the $\mathfrak{sl}_2$ operators are (anti-)self-adjoint with respect to the associated graded pairing \eqref{eq: gradedpairing}. The second proof, presented in Section~\ref{sec: stacks sl2}, works uniformly for every $d\geq 0$ by showing that the $\mathfrak{sl}_2$-triples preserve the modified Mumford relations introduced in the previous section.


In fact, we will construct two commuting $\mathfrak{sl}_2$-triples on $\gr^C_\bullet H^*(\FN_{2,1}^{\leq d})$ for each $d\geq 0$. For $d=0$, this pair of $\mathfrak{sl}_2$-triples gives rise to the $\mathbb{\mathbb{Z}}/2\times \mathbb{Z}/2$ symmetry
\begin{align*}
\overline\Omega(N_{2,1}, q, t)&=\overline \Omega(N_{2,1}, q^{-1}, t^{-1}),\\
\overline\Omega(N_{2,1}, q, t)&=\overline\Omega(N_{2,1}, q, q^{-2}t^{-1}),\\
\overline\Omega(N_{2,1}, q, t)&=\overline\Omega(N_{2,1}, q^{-1}, q^2t)\,.
\end{align*}
The diagonal $\mathfrak{sl}_2$-triple is related to the Chern grading and furthermore implies (see Corollary~\ref{cor: diagonal h is shifted Chern grading}) that the Chern grading specialization $\overline{\Omega}(N_{2,1},q,1)=\sum_{i=-n}^n a_i\, q^i$ is \textit{unimodal}, in the sense that 
\[
a_{-n}\leq \dots\leq a_0\geq\dots\geq a_{n}.
\]

Recall that the descendent algebra $\BD$ is freely generated by the symbols $\alpha, \beta, \psi_i$ for $1\leq i\leq 2g$ introduced in Remark \ref{rmk: taut class from End}. This algebra is double-graded by the cohomological degree
$$\deg(\alpha)=2,\quad \deg(\beta)=4,\quad \deg(\psi_i)=3$$
and the Chern degree
$$\deg^C(\alpha)=\deg^C(\beta)=\deg^C(\psi_i)=2. 
$$
The cohomological degree induces a super-grading $|\cdot|$ valued in $\{0,1\}$. 
All algebraic concepts such as vector spaces, operators, commutators, and (anti-)self-adjointness will be super-graded. 

We make some comments on the natural $\Sp(2g,\BZ)$-equivariant structures which will be used in this section. The group $\Gamma\coloneqq\Sp(2g,\BZ)$ acts on the cohomology ring $H^*(\Sigma,\BZ)$ via either the mapping class group modulo the Torelli group of the underlying manifold of $\Sigma$, or the monodromy group of the moduli stack of smooth algebraic curves. Explicitly, $\Gamma$ is the automorphism group of the symplectic lattice $H^1(\Sigma,\BZ)$ and acts trivially on the even part of $H^*(\Sigma,\BZ)$. 

The group $\Gamma$ also acts on the cohomology rings of the moduli spaces and stacks considered in this paper via monodromy. Since the restriction homomorphisms between such cohomology rings are $\Gamma$-equivariant, it suffices to consider the $\Gamma$-action on the descendent algebra $\BD\simeq H^*(\FN_{2,1})$ which we describe now. If $\sigma\in \Gamma$ is represented by a matrix $A$ satisfying $^\mathsf{T}\kern -2 pt A\cdot J \cdot A = J$, where $J=\begin{pmatrix}0&I_n\\-I_n&0\end{pmatrix}$ is the standard symplectic matrix, then 
\[
\sigma\cdot \alpha = \alpha,\quad \sigma\cdot \beta = \beta,\quad \sigma\cdot \psi_i = \hat{A}\cdot \psi_i=\sum_j \hat{A}_{ji}\,\psi_j
\]
where $\hat{A}\coloneqq{}^\mathsf{T}\kern -2 pt A^{-1}$. The $\Gamma$-invariant subring of the descendent algebra is 
$$\BD^\Gamma = \BQ[\alpha,\beta,\gamma]/(\gamma^{g+1})\subseteq \BD
$$
where $\gamma= - 2\sum_{i=1}^g \psi_i\psi_{i+g}$.


\subsection{Commuting $\mathfrak{sl}_2$-triples on the descendent algebra}

Using the free algebra structure of the descendent algebra $\BD$, we define the differential operators acting on $\BD$: 
\begin{align*}
    \Fe_\alpha&\coloneqq\alpha,\\
    \Fh_\alpha&\coloneqq2\alpha\frac{\partial}{\partial \alpha}+\sum_{i=1}^{2g}\psi_i\frac{\partial}{\partial \psi_i}-(g-1),\\
    \Ff_\alpha&\coloneqq
    -\alpha\frac{\partial^2}{\partial\alpha^2}
    +(g-1)\frac{\partial}{\partial \alpha}
    -\frac{\partial}{\partial \alpha}\sum_{i=1}^{2g}\psi_i\frac{\partial}{\partial \psi_i}
    -\frac{\beta}{4}\sum_{i=1}^{g}\frac{\partial}{\partial \psi_i}\frac{\partial}{\partial \psi_{i+g}}.
\end{align*}
Similarly, we define the other set of operators exchanging the role of $\alpha$ and $\beta$:
\begin{align*}
    \Fe_\beta&\coloneqq\beta,\\
    \Fh_\beta&\coloneqq2\beta\frac{\partial}{\partial \beta}+\sum_{i=1}^{2g}\psi_i\frac{\partial}{\partial \psi_i}-(g-1),\\
    \Ff_\beta&\coloneqq-\beta\frac{\partial^2}{\partial\beta^2}
    +(g-1)\frac{\partial}{\partial \beta}
    -\frac{\partial}{\partial \beta}\sum_{i=1}^{2g}\psi_i\frac{\partial}{\partial \psi_i}
    -\frac{\alpha}{4}\sum_{i=1}^{g}\frac{\partial}{\partial \psi_i}\frac{\partial}{\partial \psi_{i+g}}.
\end{align*} 

\smallskip

Recall that a triple of operators $(\Fe, \Fh, \Ff)$ acting on a vector space forms an $\mathfrak{sl}_2$-triple if 
$$[\Fe,\Ff]=\Fh,\quad [\Fe, \Fh]=2\Fe,\quad [\Ff,\Fh]=-2\Ff.
$$
We say that a pair of $\mathfrak{sl}_2$-triples $(\Fe, \Fh, \Ff)$ and $(\Fe', \Fh', \Ff')$ commutes if any operator in $(\Fe, \Fh, \Ff)$ commutes with any operator in $(\Fe', \Fh', \Ff')$, hence defining a representation of $\mathfrak{sl}_2\times \mathfrak{sl}_2$. 
\begin{prop}\label{prop: sl2 for descendent}
The operators $(\Fe_\alpha, \Fh_\alpha, \Ff_\alpha)$ and $(\Fe_\beta, \Fh_\beta, \Ff_\beta)$ form a commuting pair of $\mathfrak{sl}_2$-triples, hence defining a representation of $\mathfrak{sl}_2\times \mathfrak{sl}_2$ on $\BD$.
\end{prop}
\begin{proof}
We first show that $(\Fe_\alpha, \Fh_\alpha, \Ff_\alpha)$ form an $\mathfrak{sl}_2$-triple. Note that $\Fh_\alpha$ is a grading operator with respect to which the multiplication operators by $\alpha$, $\beta$ and $\psi_i$ are of $\Fh_\alpha$-degree $2$, $0$, and $1$, respectively. From this observation it follows that $\Fe_\alpha$ and $\Ff_\alpha$ have $\Fh_\alpha$-degree $2$ and $-2$, respectively, hence $[\Fe_\alpha, \Fh_\alpha]=2\Fe_\alpha$ and $[\Ff_\alpha,\Fh_\alpha]=-2\Ff_\alpha$. On the other hand, direct computation using the Leibniz rule of the commutator gives
\begin{align*}
    [\Ff_\alpha,\Fe_\alpha]
    &=\left[
    -\alpha\frac{\partial^2}{\partial\alpha^2}
    +(g-1)\frac{\partial}{\partial \alpha}
    -\frac{\partial}{\partial \alpha}\sum_{i=1}^{2g}\psi_i\frac{\partial}{\partial \psi_i}
    -\frac{\beta}{4}\sum_{i=1}^{g}\frac{\partial}{\partial \psi_i}\frac{\partial}{\partial \psi_{i+g}},\alpha\right]\\
    &=-2\alpha\frac{\partial}{\partial \alpha}+(g-1)-\sum_{i=1}^{2g}\psi_i\frac{\partial}{\partial \psi_i}
\end{align*}
which is equal to $-\Fh_\alpha$. The same argument shows that $(\Fe_\beta, \Fh_\beta, \Ff_\beta)$ form an $\mathfrak{sl}_2$-triple.

We now show that the pair $(\Fe_\alpha, \Fh_\alpha, \Ff_\alpha)$ and $(\Fe_\beta, \Fh_\beta, \Ff_\beta)$ commutes. All the commutators vanish for rather trivial reasons except for $[\Ff_\alpha,\Ff_\beta]=0$ which we explain below. By independence of the variables $\alpha, \beta,\psi_i$'s, many terms of $[\Ff_\alpha,\Ff_\beta]$ vanish, leaving us with
\small
\begin{align*}
    [\Ff_\alpha,\Ff_\beta]
    =&\left[-\alpha\frac{\partial^2}{\partial\alpha^2}+(g-1)\frac{\partial}{\partial \alpha}, -\frac{\alpha}{4}\right]\sum_{i=1}^{g}\frac{\partial}{\partial \psi_i}\frac{\partial}{\partial \psi_{i+g}}
    -
    \left[-\beta\frac{\partial^2}{\partial\beta^2}
+(g-1)\frac{\partial}{\partial \beta},-\frac{\beta}{4}\right]\sum_{i=1}^{g}\frac{\partial}{\partial \psi_i}\frac{\partial}{\partial \psi_{i+g}}\\
&-\left[\frac{\partial}{\partial \alpha}\sum_{i=1}^{2g}\psi_i\frac{\partial}{\partial \psi_i},-\frac{\alpha}{4}\sum_{i=1}^{g}\frac{\partial}{\partial \psi_i}\frac{\partial}{\partial \psi_{i+g}}\right]
+
\left[\frac{\partial}{\partial \beta}\sum_{i=1}^{2g}\psi_i\frac{\partial}{\partial \psi_i},-\frac{\beta}{4
}\sum_{i=1}^{g}\frac{\partial}{\partial \psi_i}\frac{\partial}{\partial \psi_{i+g}}\right]\\
=&\,\frac{1}{2}\left(\alpha\frac{\partial}{\partial\alpha}-\beta\frac{\partial}{\partial\beta}\right)\sum_{i=1}^{g}\frac{\partial}{\partial \psi_i}\frac{\partial}{\partial \psi_{i+g}}
+\frac{1}{4}\left(
\alpha\frac{\partial}{\partial\alpha}
-\beta\frac{\partial}{\partial\beta}
\right)\left[\sum_{i=1}^{2g}\psi_i\frac{\partial}{\partial \psi_i},\sum_{i=1}^{g}\frac{\partial}{\partial \psi_i}\frac{\partial}{\partial \psi_{i+g}}\right].
\end{align*}

\normalsize

Since $\sum_{i=1}^{2g}\psi_i\frac{\partial}{\partial \psi_i}$ is the grading operator assigning degree $1$ to each $\psi_i$, the second term cancels the first term.
\end{proof}

One can check that all the operators in the $\mathfrak{sl}_2$-triples are $\Gamma$-equivariant
, and in particular they preserve the $\Gamma$-invariant subring. For later use, we record how the operators simplify when restricted to $\BD^\Gamma$.
\begin{lem}\label{lem: sl2invariantsubring}
    On the $\Gamma$-invariant subring $\BD^\Gamma=\BQ[\alpha,\beta,\gamma]/(\gamma^{g+1})$, we have
$$\sum_{i=1}^{2g}\psi_i\frac{\partial}{\partial \psi_i}=2\gamma\frac{\partial}{\partial\gamma}\quad
\text{and}\quad
\sum_{i=1}^g\frac{\partial}{\partial \psi_i}\frac{\partial}{\partial \psi_{i+g}}=-2\gamma\frac{\partial^2}{\partial\gamma^2}+2g\frac{\partial}{\partial\gamma}.
$$
\end{lem}
\begin{proof}
The differential operators on the right hand sides of both equalities are well-defined on $\BQ[\alpha,\beta,\gamma]/(\gamma^{g+1})$ because they preserve the ideal $(\gamma^{g+1})$. Since all the operators in the statement annihilate the unit $1$ and commute with $\alpha$ and $\beta$, it suffices to match the commutator of both sides with the multiplication operator $\gamma$. 

Since $\sum_{i=1}^{2g}\psi_i\frac{\partial}{\partial \psi_i}$ is the grading operator assigning degree $1$ to each $\psi_i$ and $\gamma$ is quadratic in the classes $\psi_i$, we have
$$\left[\sum_{i=1}^{2g}\psi_i\frac{\partial}{\partial \psi_i},\gamma
    \right]=2\gamma=\left[2\gamma\frac{\partial}{\partial\gamma},\gamma\right],
$$
proving the first equality. The second equality also follows from direct computation.
\end{proof}

By the above lemma, the operators $(\Fe_\alpha, \Fh_\alpha, \Ff_\alpha)$ become 
\begin{align}\label{eq: invariant f operator}
    \Fe_\alpha&=\alpha,\nonumber\\
    \Fh_\alpha&=2\alpha\frac{\partial}{\partial \alpha}+2\gamma\frac{\partial}{\partial\gamma}-(g-1),\nonumber\\
    \Ff_\alpha&=
    -\alpha\frac{\partial^2}{\partial\alpha^2}
    +(g-1)\frac{\partial}{\partial \alpha}
    -2\gamma\frac{\partial^2}{\partial \alpha\partial\gamma}
    +\frac{1}{2}\beta\gamma\frac{\partial^2}{\partial\gamma^2}-\frac{g}{2}\beta\frac{\partial}{\partial\gamma}
\end{align}
when we restrict to $\BD^\Gamma$. Similar formulas hold for $(\Fe_\beta, \Fh_\beta, \Ff_\beta)$. 

\subsection{(Anti-)self-adjoint property} 

Let $V$ be a super-graded vector space. We say that an element $D\in V$ or an operator $F:V\rightarrow V$ is pure (with respect to the super-grading) if it is either even or odd.

\begin{defn} Let $V$ be a super-graded vector space with a pairing $\langle -,-\rangle:V\otimes V\rightarrow \BQ$. We say that a pure operator $F:V\rightarrow V$ is $(\pm)$-self-adjoint with respect to the given pairing if 
$$\langle F(D),D'\rangle = \pm(-1)^{|F||D|}\langle D,F(D')\rangle$$
for all pure elements $D,D'\in V$. If the sign is $+$ (resp. $-$), we simply say that the operator is self-adjoint (resp. anti-self-adjoint). 
\end{defn}

The main theorem of this section is the $(\pm)$-self-adjoint property of the operators $(\Fe_\alpha, \Fh_\alpha, \Ff_\alpha)$ and $(\Fe_\beta, \Fh_\beta, \Ff_\beta)$. We recall from Section \ref{subsec: associatedgraded} the definition of the symmetric pairing $\langle -,-\rangle^\gr:\BD\otimes \BD\rightarrow\BQ$ on the descendent algebra given by
$$\langle D,D'\rangle^\gr = \int_N^{\gr} D\cdot D'
$$
where $\displaystyle\int_N^\gr$ denotes the integration over $N$ of only the top Chern degree part of $D\cdot D'$.  

\begin{thm}\label{thm: self-adjoint} 
    The operators $\Fe_\alpha, \Fe_\beta, \Ff_\alpha, \Ff_\beta$ (resp. $\Fh_\alpha, \Fh_\beta$) are self-adjoint (resp. anti-self-adjoint) with respect to the pairing $\langle -,-\rangle^\gr$ on $\BD$. 
\end{thm}

Combining the above $(\pm)$-self-adjoint property of the operators with the perfectness of the pairing on $\gr^C_\bullet H^*(N)$ in Corollary \ref{cor: 1.4}, we obtain an $\mathfrak{sl}_2\times \mathfrak{sl}_2$-representation on $\gr^C_\bullet H^*(N)$.

\begin{cor}\label{cor: descending the operators}
    The $\mathfrak{sl}_2\times \mathfrak{sl}_2$-representation on $\BD$ descends via the surjection $\BD\twoheadrightarrow \gr^C_\bullet H^*(N)$.
\end{cor}

\begin{proof} Let $F$ be one of the operators in the commuting $\mathfrak{sl}_2$-triples. The statement of the corollary amounts to showing that the operator $F$ preserves the kernel of the surjective homomorphism $\BD\twoheadrightarrow \gr^C_\bullet H^*(N)$. By Theorem \ref{thm: omegaM21} and Corollary \ref{cor: 1.4}, the kernel $I^\gr$ of this surjection is equal to the kernel of the symmetric pairing $\langle -,-\rangle^\gr$ on $\BD$. Thus it suffices to show that $\langle D,-\rangle^\gr=0$ implies $\langle F(D),-\rangle^\gr=0$, which follows from the $(\pm)$-self-adjoint property of $F$. 
\end{proof}

We now explain how to reduce the proof of Theorem \ref{thm: self-adjoint} to showing that $\Ff_\alpha$ is self-adjoint. 

\begin{lem}\label{lem: swapping alpha and beta}
    The pairing $\langle-,-\rangle^\gr$ is preserved under the algebra automorphism of $\BD$ swapping $\alpha$ and $\beta$.
\end{lem}
\begin{proof}
Recall that the pairing $\langle-,-\rangle^\gr$ on $\BD$ uses the top Chern degree integral over $N$. Note that a monomial in the descendent algebra
\[D=\alpha^{n}\beta^{m}\prod_{i=1}^{2g}\psi_{i}^{p_{i}}\in \BD\]
has the top cohomological degree and top Chern degree {only} if 
$$2n+4m+3\sum_{i=1}^{2g}p_i=6g-6,\quad 2n+2m+2\sum_{i=1}^{2g}p_i=4g-4.
$$
The two equalities in particular imply that $n=m$, hence proving the lemma.
\end{proof}

 Since $(\Fe_\alpha, \Fh_\alpha, \Ff_\alpha)$ and $(\Fe_\beta, \Fh_\beta, \Ff_\beta)$ are related by swapping the role of $\alpha$ and $\beta$, the above lemma shows that it suffices to consider $(\Fe_\alpha, \Fh_\alpha, \Ff_\alpha)$. On the other hand, a multiplication operator $\Fe_\alpha$ is clearly self-adjoint. The grading operator $\Fh_\alpha$ is anti-self-adjoint due to degree reason as we explain below.

\begin{prop}\label{prop: h is anti-self-adjoint}
     The operator $\Fh_\alpha$ is anti-self-adjoint with respect to $\langle-,-\rangle^\gr$.
\end{prop}
\begin{proof}
Given any two monomials
$$D=\alpha^n\beta^m\prod_{i=1}^{2g}\psi_i^{p_i},\quad D'=\alpha^{n'}\beta^{m'}\prod_{i=1}^{2g}\psi_i^{p'_i},
$$
we need to show that $\langle \Fh_\alpha(D),D'\rangle^\gr=-\langle D,\Fh_\alpha(D')\rangle^\gr$. Equivalently, we need to show
\begin{equation}\label{eq: vanishing with h operator}
    \int_N^\gr \Fh_\alpha(D)\cdot D'+D\cdot \Fh_\alpha(D')=0.
\end{equation}
Since 
$$\Fh_\alpha=2\alpha\frac{\partial}{\partial \alpha}+\sum_{i=1}^{2g}\psi_i\frac{\partial}{\partial \psi_i}-(g-1)
$$
is the operator preserving the cohomological degree and Chern degree, the integral on the left hand side of \eqref{eq: vanishing with h operator} is trivially zero unless 
$$\deg(D)+\deg(D')=6g-6,\quad \deg^C(D)+\deg^C(D')=4g-4.
$$
The above two equalities amount to $n+n'=m+m'$ and $2(n+n')+\sum_{i=1}^{2g}(p_i+p'_i)=2g-2$, which we assume from now. Since $\Fh_\alpha$ is a grading operator with all monomials being eigenvectors, the equality \eqref{eq: vanishing with h operator} follows from
\[
\Fh_\alpha(D)\cdot D'+D\cdot \Fh_\alpha(D')=\left(2(n+n')+\sum_{i=1}^{2g}(p_i+p'_i)-(2g-2)
\right)D\cdot D'=0. \qedhere
\]

\end{proof}

\subsection{Self-adjointness of $\mathfrak{f}_\alpha$} We now focus on proving that $\mathfrak{f}_\alpha$ is self-adjoint, which is the heart of Theorem \ref{thm: self-adjoint}. The following lemma will be repeatedly used.

\begin{lem}\label{lem: reducing self-adjoint}
Let $V$ be a super-graded vector space with a pairing $\langle-,-\rangle$. Suppose that $V$ is spanned by an even vector $\mathbf{1}\in V$ acted upon by a collection of self-adjoint operators $\{G_s\}_{s\in S}$. Then an operator $F:V\rightarrow V$ is $(\pm)$-self-adjoint if and only if the following two conditions are satisfied:
\begin{enumerate}
    \item[(i)] $\langle F(\mathbf{1}),D\rangle=\pm\langle \mathbf{1},F(D)\rangle$ for all $D\in V$ and
    \item[(ii)] $[F,G_s]$ is $(\mp)$-self-adjoint for all $s\in S$.
\end{enumerate}
\end{lem}
\begin{proof}  
The ``only if" direction is straightforward to check so we only prove the ``if" direction. Assume that a pure operator $F$ satisfies the conditions in (i) and (ii). For each $n\in \BZ_{\geq 0}$, define a subspace 
$$V_n\coloneqq\text{span}\{G_{s_1}\circ\cdots\circ G_{s_n}(\mathbf{1})\,|\,s_1,\dots,s_n\in S\}\subseteq V. 
$$
By assumption, the filtration $V_0\subseteq V_1\subseteq \cdots \subseteq V$ saturates $V$ in the sense that any $D\in V$ lies in $V_n$ for some $n\geq 0$. Therefore it suffices to show that for every $n\geq 0$ we have
\begin{equation}\label{eq: induction argument}
    \langle F(D),D'\rangle = \pm (-1)^{|F||D|}\langle D,F(D')\rangle\quad \text{for all}\quad D\in V_n,\ D'\in V.
\end{equation}
We prove this by induction on $n\geq 0$. Since $V_0=\text{span}\{\mathbf{1}\}$, the $n=0$ case is precisely the statement of (i). Now let $n>0$ and assume \eqref{eq: induction argument} for the $n-1$ case. Given $D=G_{s_1}\circ\cdots\circ G_{s_n}(\mathbf{1})\in V_n$ and $D'\in V$, we can write $D=G_{s_1}(D'')$ where $D''=G_{s_2}\circ\cdots\circ G_{s_n}(\mathbf{1})\in V_{n-1}$. Using the self-adjoint property of $G_{s_1}$ and $(\mp)$-self-adjoint property of $[F,G_{s_1}]$, we have
\begin{align*}
    \langle F(D),D'\rangle 
    &=\langle F\circ G_{s_1}(D''),D'\rangle\\
    &=\langle [F,G_{s_1}](D''),D'\rangle+(-1)^{|F||G_{s_1}|}\langle G_{s_1}\circ F(D''),D'\rangle\\
    &=\mp(-1)^{(|F|+|G_{s_1}|)|D''|}\langle D'',[F,G_{s_1}]D'\rangle+(-1)^{|G_{s_1}||D''|}\langle F(D''),G_{s_1}(D')\rangle. 
\end{align*}
Since $D''\in V_{n-1}$, the induction hypothesis allows us to rewrite the second term to get
\begin{align*}
    \langle F(D),D'\rangle 
    &=\mp(-1)^{(|F|+|G_{s_1}|)|D''|}\langle D'',[F,G_{s_1}]D'\rangle\pm(-1)^{(|F|+|G_{s_1}|)|D''|}\langle D'',F\circ G_{s_1}(D')\rangle\\
    &=\pm (-1)^{(|F|+|G_{s_1}|)|D''|}\left(
    \langle D'', F\circ G_{s_1}(D')- [F, G_{s_1}]D'\rangle
    \right)\\
    &=\pm (-1)^{(|F|+|G_{s_1}|)|D''|} (-1)^{|F||G_{s_1}|}\langle D'', G_{s_1}\circ F(D')\rangle\\
    &=\pm (-1)^{|F|(|D''|+|G_{s_1}|)}\langle G_{s_1}(D''),F(D')\rangle\\
    &=\pm (-1)^{|F||D|}\langle D,F(D')\rangle. \qedhere
\end{align*}

\end{proof}

We apply the above lemma when the super-graded vector space is $\BD$ equipped with the pairing $\langle-,-\rangle^\gr$, the even vector is the unit $1$ and the collection of self-adjoint operators is the set of multiplication operators $\{\alpha,\beta,\psi_1,\dots,\psi_{2g}\}$. According to the lemma, in order to show that $\Ff_\alpha$ is self-adjoint, it suffices to show that 
\begin{enumerate}
    \item[(i)] $\langle \Ff_\alpha(1),D\rangle^\gr=\langle 1,\Ff_\alpha(D)\rangle^\gr$ for all $D\in \BD$,
    \item[(ii)] $[\Ff_\alpha,\alpha]$, $[\Ff_\alpha,\beta]$, $[\Ff_\alpha,\psi_i]$ for $i=1,\dots,2g$ are all anti-self-adjoint.
\end{enumerate}
Since $\Ff_\alpha(1)=0$, part (i) is equivalent to
$$\int_N^\gr \Ff_\alpha(D)=0\quad\text{for all \,} D\in \BD.
$$
On the other hand, since $[\Ff_\alpha,\alpha]=-\Fh_\alpha$ is anti-self-adjoint by Proposition \ref{prop: h is anti-self-adjoint} and $[\Ff_\alpha,\beta]=0$ by Proposition \ref{prop: sl2 for descendent}, the only thing we are left with in part (ii) is to show that $[\Ff_\alpha,\psi_i]$ is anti-self-adjoint for all $1\leq i\leq 2g$. These two statements will be proved in Proposition \ref{prop: vanishing integral of f} and Proposition \ref{prop: anti-self-adjoint of [f,psi]}, respectively.

\smallskip

We briefly explain the main ingredients in the proofs of Proposition~\ref{prop: vanishing integral of f} and Proposition~\ref{prop: anti-self-adjoint of [f,psi]}. Both propositions 
can be checked in principle if we understand all integral values of the top Chern degree monomials $D\in \BD$. It turns out that we only need some proportionalities between such integrals coming from the Virasoro constraints and monodromy invariance with respect to the group $\Gamma=\Sp(2g,\BZ)$. 

The Virasoro constraints of the moduli space $N$, as written in \cite[Example 2.24]{BLM}, give
\begin{equation}\label{eq: Virasoro}
    (g-p)\int_N \alpha^n\beta^m\gamma^p = -2n\int_N \alpha^{n-1}\beta^{m-1}\gamma^{p+1}.
\end{equation}
Note that this is also an immediate consequence of \cite[(30)]{Tha1}. On the other hand, monodromy invariance of the integrals over $N$ implies the following two statements \cite[pp. 143--144]{Tha1}. First, we have
$$\int_N \alpha^n\beta^m\prod_{i=1}^{2g}\psi_i^{p_i}= 0
$$
unless $p_i=p_{i+g}$ for $1\leq i\leq g$. Setting $\gamma_i\coloneqq\psi_i\psi_{i+g}$ for $1\leq i\leq g$, the integral value
$$\int_N \alpha^n\beta^m\gamma_{i_1}\cdots\gamma_{i_p}
$$
depends on $1\leq i_1<\cdots<i_p\leq g$ only through $p$. Since $\gamma=-2\sum_{i=1}^g\gamma_i$, this implies
\begin{equation}\label{eq: monodromy invariance}
    \int_N \alpha^n\beta^m\gamma^p = (-2)^p\cdot g(g-1)\cdots(g-p+1)\int_N \alpha^n\beta^m\gamma_{i_1}\cdots\gamma_{i_p}.
\end{equation}


\begin{lem}\label{lem: Lie group}
    Let $V$ be a complex representation of a complex reductive group $G$. Suppose that $V$ is saturated by finite dimensional $G$-invariant subspaces. If $\phi:V\rightarrow\BC$ is a $G$-invariant linear map such that $\phi\big|_{V^G}=0$, then $\phi=0$.
\end{lem}
\begin{proof}
Since $V$ is saturated by finite dimensional $G$-invariant subspaces, we can reduce the proof to the case where $V$ is finite dimensional. Let $K\subset G$ be a maximal compact subgroup of $G$ and $d\mu$ be the unique Haar measure on $K$. Define the ``averaging map" by
$$I:V\rightarrow V,\quad x\mapsto \int_K k\cdot x\, d\mu. 
$$
Then $I$ is a complex linear map with the property that $I(V)\subseteq V^K=V^G$; the fact that $V^K=V^G$ follows from Zariski density of $K\subset G$, which is sometimes known as ``Weyl Unitarian Trick'' (see \cite[Section 2.2.2]{Springer}). On the other hand, for $x\in V$ we have 
$$\phi(I(x))=\int_K \phi(k\cdot x)\,d\mu=\int_K \phi(x)\,d\mu=\phi(x).
$$
Since $I(x)\in V^G$ and $\phi\big|_{V^G}=0$, we conclude that $\phi(x)=0$. 
\end{proof}

\begin{prop}\label{prop: vanishing integral of f}
    We have $\displaystyle\int_N^\gr \Ff_\alpha(D)=0$ for all $D\in \BD$.
\end{prop}
\begin{proof}
    We first prove the statement for $D\in \BD^\Gamma$. Since $\Ff_\alpha$ is of cohomological degree $-2$ and Chern degree $-2$, it suffices to consider monomials $D=\alpha^n\beta^m\gamma^p$ with
    \begin{align*}
        \deg(D)&=2n+4m+6p=(6g-6)+2,\\
        \deg^C(D)&=2n+2m+4p=(4g-4)+2.
    \end{align*}
The two equalities are equivalent to $n=m+1$ and $m+p=g-1$. By applying the formula \eqref{eq: invariant f operator} of $\Ff_\alpha$ to $D=\alpha^{m+1}\beta^m\gamma^p$, we have
\begin{align*}
    \Ff_\alpha(D)
    =&-(m+1)m\alpha^{m}\beta^m\gamma^p+(g-1)(m+1)\alpha^{m}\beta^m\gamma^p-2(m+1)p\alpha^m\beta^m\gamma^{p}\\
    &+\frac{1}{2}p(p-1)\alpha^{m+1}\beta^{m+1}\gamma^{p-1}-\frac{g}{2}p\alpha^{m+1}\beta^{m+1}\gamma^{p-1}\\
    =&\,(m+1)(-m+g-1-2p)\alpha^m\beta^m\gamma^p
    +\frac{p}{2}(p-1-g)\alpha^{m+1}\beta^{m+1}\gamma^{p-1}. 
\end{align*}
On the other hand, by the proportionality \eqref{eq: Virasoro} coming from the Virasoro constraints, the integral $\int_N \Ff_\alpha(D)$ is zero if 
$$(m+1)(-m+g-1-2p)(g-p+1)+\frac{p}{2}(p-1-g)(-2(m+1))=0.
$$
This equality follows from $m+p=g-1$.

So far, we have shown that the $\Gamma$-invariant map 
$$\phi:\BD\rightarrow\BQ,\quad D\mapsto \int_N^\gr \Ff_\alpha(D)
$$
vanishes once we restrict the domain of $\phi$ to $\BD^\Gamma$. Note that we can extend scalars to the complex numbers in the following sense. The complex reductive group $\Gamma_\BC\coloneqq\Sp(2g,\BC)$ acts on $\BD_\BC\coloneqq\BD\otimes_\BQ \BC$ so that $\phi_\BC:\BD_\BC\rightarrow \BC$ is $\Gamma_\BC$-invariant. By Lemma \ref{lem: Lie group}, we conclude that $\phi_\BC=0$ from the fact that $\phi_\BC$ restricts to zero on the smaller domain $\BD_\BC^{\Gamma_\BC}=\BD^\Gamma\otimes_\BQ\BC$. The condition on the saturation by finite dimensional invariant subspaces in the lemma is satisfied for example by the Chern filtration. 
\end{proof}

\begin{prop}\label{prop: anti-self-adjoint of [f,psi]}
    For $1\leq i\leq 2g$, the operator $[\Ff_\alpha,\psi_i]$ is anti-self-adjoint with respect to $\langle-,-\rangle^\gr$.
\end{prop}
\begin{proof}
By Lemma \ref{lem: reducing self-adjoint}, $[\Ff_\alpha,\psi_j]$ is anti-self-adjoint if and only if
\begin{enumerate}
    \item[(i)] $\langle [\Ff_\alpha,\psi_j](1),D\rangle^\gr=-\langle 1,[\Ff_\alpha,\psi_j](D)\rangle^\gr$ for all $D\in \BD$,
    \item[(ii)] $\big[[\Ff_\alpha,\psi_j],\alpha\big]$, $\big[[\Ff_\alpha,\psi_j],\beta\big]$, $\big[[\Ff_\alpha,\psi_j],\psi_i\big]$ for $i=1,\dots,2g$ are all self-adjoint.
\end{enumerate}
Since $\Ff_\alpha$ is a quadratic differential operator, it becomes a multiplication operator after
taking commutator twice with multiplication operators. Therefore, part (ii) follows immediately. So we are left with part (i). 

Now we let $1\leq j\leq g$ and consider part (i) for $[\Ff_\alpha,\psi_j]$. The case of $[\Ff_\alpha,\psi_{j+g}]$ is essentially the same and we leave the details to the reader. We start by computing the commutator 
\begin{align}\label{eq: commutator computation}
    [\Ff_\alpha,\psi_j]
    &=\left[
    -\alpha\frac{\partial^2}{\partial\alpha^2}
    +(g-1)\frac{\partial}{\partial \alpha}
    -\frac{\partial}{\partial \alpha}\sum_{i=1}^{2g}\psi_i\frac{\partial}{\partial \psi_i}
    -\frac{\beta}{4}\sum_{i=1}^{g}\frac{\partial}{\partial \psi_i}\frac{\partial}{\partial \psi_{i+g}},\psi_j
    \right]\nonumber\\
    &=-\frac{\partial}{\partial \alpha}\psi_j
    +\frac{\beta}{4}\frac{\partial}{\partial \psi_{j+g}}. 
\end{align}
Since $[\Ff_\alpha,\psi_j](1)=0$, part (i) is equivalent to 
\begin{equation}\label{eq: second vanishing}
\int_N^\gr[\Ff_\alpha,\psi_j](D)=0\quad\text{for all}\quad D\in \BD.
\end{equation}

In order to show the vanishing \eqref{eq: second vanishing}, we follow a similar strategy as in the proof of Proposition~\ref{prop: vanishing integral of f}. One difference is that the top Chern degree integration map 
$$\BD\rightarrow\BQ,\quad D\mapsto \int_N^\gr[\Ff_\alpha,\psi_j](D)
$$
is no longer $\Gamma=\Sp(2g,\BZ)$-invariant. However, it is invariant with respect to the subgroup $\widetilde{\Gamma}\subset \Gamma$ fixing $\psi_j$ and $\psi_{j+g}$. Abstractly, $\widetilde{\Gamma}$ is isomoprhic to $\Sp(2g-2,\BZ)$, and the $\widetilde{\Gamma}$-invariant subring is generated by $\alpha, \beta, \gamma, \psi_j, \psi_{j+g}$.


By Lemma \ref{lem: Lie group}, it suffices to show the vanishing \eqref{eq: second vanishing} for the monomials 
\[
D=\alpha^n\beta^m\gamma^p\psi_j^{p_j}\psi_{j+g}^{p_{j+g}}.
\] 
Since $[\Ff_\alpha,\psi_j]$ has cohomological degree $1$ and Chern degree $0$, we may assume
\begin{align*}
        \deg(D)&=2n+4m+6p+3(p_j+p_{j+g})=(6g-6)-1,\\
        \deg^C(D)&=2n+2m+4p+2(p_j+p_{j+g})=4g-4.
    \end{align*}
In particular, we have $n=m+1$ and $p_j+p_{j+g}=1$. By the formula of $[\Ff_\alpha,\psi_j]$ in \eqref{eq: commutator computation}, $[\Ff_\alpha,\psi_j](D)=0$ if $p_j=1$ and $p_{j+g}=0$. So we may assume that 
$$D=\alpha^{m+1}\beta^{m}\gamma^p\psi_{j+g}
$$
with $m+p=g-2$. 

Recall the definition of $\gamma=-2\sum_{i=1}^g\gamma_i$ where $\gamma_i=\psi_i\psi_{i+g}$. By computation, we have
\begin{align*}
    [\Ff_\alpha,\psi_j](D)
    &=\left(-\frac{\partial}{\partial \alpha}\psi_j
    +\frac{\beta}{4}\frac{\partial}{\partial \psi_{j+g}}\right)\alpha^{m+1}\beta^m\gamma^p\psi_{j+g}\\
    &=-(m+1)\alpha^m\beta^m\gamma^p\gamma_j+\frac{1}{4}\alpha^{m+1}\beta^{m+1}\frac{\partial}{\partial\psi_{j+g}}(\gamma^p\psi_{j+g})\\
    &=-(m+1)\alpha^m\beta^m\gamma^p\gamma_j+\frac{1}{4}\alpha^{m+1}\beta^{m+1}(-2)^p(\gamma_1+\cdots \widehat{\gamma_j}+\cdots \gamma_g)^p
\end{align*}
where we used $\gamma_j\cdot \psi_{j+g}=0$ in the last step. By the proportionality \eqref{eq: monodromy invariance} coming from the monodromy invariance, we have
$$\int_N\alpha^m\beta^m\gamma^p\gamma_j= -\frac{1}{2g} \int_N \alpha^m\beta^m\gamma^{p+1}
$$
and 
$$\int_N\alpha^{m+1}\beta^{m+1}(-2)^p(\gamma_1+\cdots \widehat{\gamma_j}+\cdots \gamma_g)^p=\frac{g-p}{g} \int_N \alpha^{m+1}\beta^{m+1}\gamma^{p}. 
$$
So the vanishing \eqref{eq: second vanishing} reduces to showing
$$\int_N \Big((m+1)\frac{1}{2g}\alpha^m\beta^m\gamma^{p+1} + \frac{1}{4}\frac{g-p}{g}\alpha^{m+1}\beta^{m+1}\gamma^p
\Big)=0. 
$$
This follows from the proportionality \eqref{eq: Virasoro} coming from the Virasoro constraints because
\[
(m+1)\frac{1}{2g}(g-p)+\frac{1}{4}\frac{g-p}{g}(-2(m+1))=0. \qedhere
\]
\end{proof}

Therefore, we have completed the proof of Theorem \ref{thm: self-adjoint}.

\subsection{Diagonal $\mathfrak{sl}_2$-triple and Chern filtration}
\label{sec: diagonal sl2}

In previous sections, we have constructed a commuting pair of $\mathfrak{sl}_2$-triples $(\Fe_\alpha, \Fh_\alpha, \Ff_\alpha)$ and $(\Fe_\beta, \Fh_\beta, \Ff_\beta)$ acting on $\BD$ which descends via the surjection $\BD\twoheadrightarrow \gr^C_\bullet H^*(N)$. On the other hand, there is a diagonal Lie subalgebra $\mathfrak{sl}_2\subset \mathfrak{sl}_2\times \mathfrak{sl}_2$ corresponding to 
$$\Fe\coloneqq\Fe_\alpha+\Fe_\beta,\quad \Fh\coloneqq\Fh_\alpha+\Fh_\beta,\quad \Ff\coloneqq\Ff_\alpha+\Ff_\beta.
$$
In this section, we discuss applications of this diagonal $\mathfrak{sl}_2$-triple to the Chern filtration and the associated graded ideal $I^\gr\coloneqq\ker(\BD\twoheadrightarrow \gr^C_\bullet H^*(N))$. 

\begin{cor}\label{cor: diagonal h is shifted Chern grading}
    The $\Fh$ operator in the diagonal $\mathfrak{sl}_2$-action on $\gr^C_\bullet H^*(N)$ is the shifted Chern degree operator. In particular, $\overline{\Omega}(N,q,1)$ satisfies the symmetry and unimodality.\footnote{We note, however, that the $\mathfrak{sl}_2$-action does not give an independent proof of the symmetry since Corollary~\ref{cor: descending the operators} relies on the perfectness of the pairing $\langle -, - \rangle^\mathrm{gr}$.} 
\end{cor}
\begin{proof}
The $\Fh$ operator of the diagonal $\mathfrak{sl}_2$-triple is given by the formula
$$\Fh=2\alpha\frac{\partial}{\partial \alpha}+2\beta\frac{\partial}{\partial \beta}+2\sum_{i=1}^{2g}\psi_i\frac{\partial}{\partial \psi_i}-(2g-2).
$$
Therefore, the $\Fh$-degrees of the multiplication operators $\alpha$, $\beta$ and $\psi_i$ are precisely $2$, equal to their Chern degrees. Furthermore, the $2g-2$ in the formula of $\Fh$ is precisely half of the top Chern degree $4g-4$ of the moduli space $N=N_{2,1}$. The second statement follows from general theory of finite dimensional $\mathfrak{sl}_2$-representations. 
\end{proof}

For moduli theory in various settings, describing a complete set of tautological relations is a fundamental problem. In the case of the moduli space $N=N_{2,1}$, the ideal of tautological relations is defined as
$$I=\ker\big(\BD\twoheadrightarrow H^*(N)\big)
$$
and is characterized by the Mumford conjecture, proven by Kirwan \cite{Kir}. In our setting, we can ask the same question regarding the Chern filtration by considering the ideal
$$I^\gr\coloneqq\ker\big(\BD\twoheadrightarrow \gr^C_\bullet H^*(N)\big). 
$$
This amounts to understanding the shape of the highest Chern degree part of each tautological relation. The next corollary characterizes the ideal $I^\gr$ purely in terms of the $\Ff$ operator 
\begin{equation}\label{eq: f operator}
    -\alpha\frac{\partial^2}{\partial\alpha^2}
-\beta\frac{\partial^2}{\partial\beta^2}
+(g-1)\left(\frac{\partial}{\partial \alpha}+\frac{\partial}{\partial \beta}\right)
-\left(\frac{\partial}{\partial \alpha}+\frac{\partial}{\partial \beta}\right)\sum_{i=1}^{2g}\psi_i\frac{\partial}{\partial \psi_i}
-\frac{\alpha+\beta}{4}\sum_{i=1}^{g}\frac{\partial}{\partial \psi_i}\frac{\partial}{\partial \psi_{i+g}}
\end{equation}
in the diagonal $\mathfrak{sl}_2$-triple.
\begin{cor}\label{cor: f determines I gr}
    The associated graded ideal $I^\gr\subset \BD$ is the smallest ideal containing $C_{>4g-4}\BD$ which is closed under the $\Ff$ operator \eqref{eq: f operator}.
\end{cor}
\begin{proof}
Let $J \subset \mathbb{D}$ be the smallest ideal that contains $C_{>4g-4}\BD$ and is closed under the $\Ff$ operator \eqref{eq: f operator}. Since $I^\gr$ contains $C_{>4g-4}\BD$ by Theorem \ref{thm: topcherndegree} and is closed under the $\Ff$ operator by Corollary~\ref{cor: descending the operators}, we have $J\subseteq I^\gr$. Consider the surjections
\begin{equation}\label{eq: surjections}
\BD\twoheadrightarrow \BD/J\twoheadrightarrow \BD/I^\gr. 
\end{equation}
In order to show that $J=I^\gr$, it suffices to show the numerical statement
\begin{equation}\label{eq: need to show}
    \dim \gr^C_i (\BD/J)=\dim \gr^C_i(\BD/I^\gr)
\end{equation}
for all $i\in \BZ_{\geq 0}$. Since the diagonal $\mathfrak{sl}_2$-action descends to $\BD/J$ and $\BD/I^\gr$ with the $\Fh$ operator being the shifted Chern grading by Corollary \ref{cor: diagonal h is shifted Chern grading}, we have symmetries
$$\dim\gr_i^C(\BD/J)=\dim\gr_{4g-4-i}^C(\BD/J),\quad \dim\gr_i^C(\BD/I^\gr)=\dim\gr_{4g-4-i}^C(\BD/I^\gr).
$$
So it suffices to show the equality \eqref{eq: need to show} for $0\leq i\leq 2g-2$. 

Theorem \ref{thm: main} computes the dimensions of the Chern grading of $\BD/I^\gr\simeq \gr^C_\bullet H^*(N)$ as 
\begin{equation}\label{eq: Chern specialization}
    \sum_{i\geq 0} \dim\gr^C_i (\BD/I^\gr) q^i=\Omega(N,q,1)=\frac{(1+q^2)^{2g}}{(1-q^2)^2}-q^{2g}\cdot \frac{2^{2g}}{(1-q^2)^2}. 
\end{equation}
On the other hand, the descendent algebra $\BD$ is generated by two even generators with $\deg^C(\alpha)=\deg^C(\beta)=2$ and $2g$ odd generators with $\deg^C(\psi_i)=2$, and hence 
$$\sum_{i\geq 0} \dim \gr^C_i \,\BD=\frac{(1+q^2)^{2g}}{(1-q^2)^2}
$$
which is precisely the first term of \eqref{eq: Chern specialization}. Since the two formulas differ by the second term which starts from $q^{2g}$, $\BD\twoheadrightarrow \BD/I^\gr$ is an isomorphism up to Chern degree $2g-1$. Using the successive surjections in \eqref{eq: surjections}, we conclude the desired equality \eqref{eq: need to show} for $0\leq i\leq 2g-1$ (which is slightly larger than $0\leq i\leq 2g-2$). 
\end{proof}

\subsection{Commuting $\mathfrak{sl}_2$-triples on intermediate stacks} 
\label{sec: stacks sl2}

In Section \ref{sec: intermediatestacks} we studied the cohomology of the intermediate stacks $\FN_{2,1}^{\leq d}$ and their Chern filtration. Recall that when $d=0$, the rings $H^\ast(N_{2,1})\simeq H^\ast(\FN_{2,1}^{\leq 0})$ are isomorphic. It turns out that there exist $(\mathfrak{sl}_2\times \mathfrak{sl}_2)$-representations on the associated graded $\gr^C_\bullet H^\ast(\FN_{2,1}^{\leq d})$ of all of these stacks, generalizing the representation discussed before. 

We define the following operators on the descendent algebra $\BD$:
\begin{align*}
    \Fe_\alpha^d&\coloneqq\alpha,\\
    \Fh_\alpha^d&\coloneqq2\alpha\frac{\partial}{\partial \alpha}+\sum_{i=1}^{2g}\psi_i\frac{\partial}{\partial \psi_i}-(g+2d-1),\\
    \Ff_\alpha^d&\coloneqq
    -\alpha\frac{\partial^2}{\partial\alpha^2}
    +(g+2d-1)\frac{\partial}{\partial \alpha}
    -\frac{\partial}{\partial \alpha}\sum_{i=1}^{2g}\psi_i\frac{\partial}{\partial \psi_i}
    -\frac{\beta}{4}\sum_{i=1}^{g}\frac{\partial}{\partial \psi_i}\frac{\partial}{\partial \psi_{i+g}}\,,
\end{align*}
and
\begin{align*}
    \Fe_\beta^d&\coloneqq\beta,\\
    \Fh_\beta^d&\coloneqq2\beta\frac{\partial}{\partial \beta}+\sum_{i=1}^{2g}\psi_i\frac{\partial}{\partial \psi_i}-(g+2d-1),\\
    \Ff_\beta^d&\coloneqq-\beta\frac{\partial^2}{\partial\beta^2}
    +(g+2d-1)\frac{\partial}{\partial \beta}
    -\frac{\partial}{\partial \beta}\sum_{i=1}^{2g}\psi_i\frac{\partial}{\partial \psi_i}
    -\frac{\alpha}{4}\sum_{i=1}^{g}\frac{\partial}{\partial \psi_i}\frac{\partial}{\partial \psi_{i+g}}\,.
\end{align*}
The only difference between these operators and our previous ones (which agree when $d=0$) is the factor $g+2d-1$ replacing $g-1$. The proof of Proposition \ref{prop: sl2 for descendent} immediately generalizes to show that $(\Fe_\alpha^d, \Fh_\alpha^d, \Ff_\alpha^d)$ and $(\Fe_\beta^d, \Fh_\beta^d, \Ff_\beta^d)$ form two commuting $\mathfrak{sl}_2$-triples. Our final result is the following:

\begin{thm}\label{thm: sl2intermediatestacks}
For every $d\geq 1$, the $\mathfrak{sl}_2$-triples $(\Fe_\alpha^d, \Fh_\alpha^d, \Ff_\alpha^d)$ and $(\Fe_\beta^d, \Fh_\beta^d, \Ff_\beta^d)$ descend via the surjections $\BD\to \gr^C_\bullet H^\ast(\FN_{2,1}^{\leq d})$.
\end{thm}
The fact that the $\Fe$ operators descend is trivial. The $\Fh$ operators are affine combinations of the cohomological grading operator and the Chern grading operator, so they also descend. The non-trivial part of the theorem concerns the $\Ff$ operators. 

When $d>0$, we no longer have an integration map nor a non-degenerate pairing, so the strategy previously used to show self-adjointness of the $\Ff$ operators is not available. Instead, we use the results from Section \ref{sec: intermediatestacks}, and in particular the explicit basis \eqref{eq: relationsIgr} for the ideal
\[I^\textup{gr}_{\leq d}=\ker\big(\BD\to \gr^C_\bullet H^\ast(\FN_{2,1}^{\leq d})\big)\,.\]
The claim of Theorem \ref{thm: sl2intermediatestacks} is equivalent to $\Ff_\alpha^d,\, \Ff_\beta^d$ preserving the ideal $I^\textup{gr}_{\leq d}$. Note also that the ideals $I^\textup{gr}_{\leq d}$ are no longer symmetric under swapping $\alpha$ and $\beta$ when $d>0$ (cf. Lemma \ref{lem: swapping alpha and beta}), so the statements for
 $\Ff_\alpha^d$ and $\Ff_\beta^d$ have to be shown separately. 

 We first prove the following lemma that will help us calculate the action of the $\Ff$ operators on the relations \eqref{eq: relationsIgr}.  

\begin{lem}\label{lem: commutatorfsigma}
    Let $\sigma_l\in \mathsf{Prim}_l\subseteq \BD$. Then, the restriction of the commutators $[\Ff_\alpha^d, \sigma_l]$ and $[\Ff_\beta^d, \sigma_l]$ to the $\Gamma$-invariant subring $\BD^\Gamma$ are
    \begin{align*}
[\Ff_\alpha^d, \sigma_l]_{|\BD^\Gamma}&=-l\sigma_l\frac{\partial}{\partial \alpha}+\frac{l}{2}\beta\sigma_l\frac{\partial}{\partial \gamma}\\
[\Ff_\beta^d, \sigma_l]_{|\BD^\Gamma}&=-l\sigma_l\frac{\partial}{\partial \beta}+\frac{l}{2}\alpha\sigma_l\frac{\partial}{\partial \gamma}\,.
    \end{align*}
\end{lem}
\begin{proof}
    Since $\Ff_\alpha^d,\, \Ff_\beta^d$ are $\Gamma$-equivariant and $\Gamma$ acts transitively on $\mathsf{Prim}_l$, we can assume without loss of generality that $\sigma_l=\psi_1\ldots \psi_l$. We have
    \[[\Ff_\alpha^d, \psi_1\ldots\psi_l]=-\frac{\partial}{\partial \alpha}\left[\sum_{i=1}^{2g}\psi_i\frac{\partial}{\partial \psi_{i}}, \psi_1\ldots\psi_l\right]-\frac{\beta}{4}\left[\sum_{i=1}^{g}\frac{\partial}{\partial \psi_i}\frac{\partial}{\partial \psi_{i+g}}, \psi_1\ldots \psi_l\right]\]
The first commutator on the right hand side is clearly $l\sigma_l$. The second commutator is
\[\left[\sum_{i=1}^{g}\frac{\partial}{\partial \psi_i}\frac{\partial}{\partial \psi_{i+g}}, \psi_1\ldots \psi_l\right]=-\sum_{i=1}^l \psi_1\ldots\psi_{i-1}\frac{\partial}{\partial \psi_{i+g}}\psi_{i+1}\ldots \psi_l\,.\]
This is a derivation on $\BD$, so it is enough to determine the images of $\alpha, \beta,\gamma$ under it to calculate its restriction to $\BD^\Gamma$. Clearly, $\alpha$ and $\beta$ are annihilated, so it remains to find 
\[\left[\sum_{i=1}^{g}\frac{\partial}{\partial \psi_i}\frac{\partial}{\partial \psi_{i+g}}, \psi_1\ldots \psi_l\right]\gamma=2\sum_{i=1}^l \psi_1\ldots \psi_{i-1}\frac{\partial}{\partial \psi_{i+g}}\psi_{i+1}\ldots \psi_{l}\cdot \psi_{i}\cdot \psi_{i+g}=-2l\sigma_l\,.\]
The proof of the formula for $[\Ff_\alpha^d, \sigma_l]_{|\BD^\Gamma}$ follows. The exact same calculation applies to $[\Ff_\beta^d, \sigma_l]_{|\BD^\Gamma}$ by switching $\alpha$ and $\beta$. 
\end{proof}

\begin{proof}[Proof of Theorem \ref{thm: sl2intermediatestacks}]
We denote
\[
R_{k,m,l}\coloneqq\sum_{\substack{b+c=m\\a+b+2c=k-g-l}}(g-l-c)!\frac{\alpha^a}{a!}\frac{\beta^b}{b!}\frac{(2\gamma)^c}{c!}\in \BD^\Gamma.
\]
As explained in Section \ref{subsec: relationsIgr}, the ideal $I^{\gr}_{\leq d}$ is freely generated as a $\BQ[\beta]$-module by the relations $R_{k,m,l}\sigma_l$ for $k,m,l$ such that $k\geq 2g+2d$ and $l+m\leq g$, with $\sigma_l$ ranging in a basis of $\mathsf{Prim}_l$. Since $[\Ff_\alpha^d, \beta]=0$ and $[\Ff_\beta^d, \beta]=-\Fh_\beta^d$ both preserve the ideal $I^{\gr}_{\leq d}$, it is enough to show that 
\[\Ff_\alpha^d(R_{k,m,l}\sigma_l),\: \Ff_\beta^d(R_{k,m,l}\sigma_l)\in I^{\textup{gr}}_{\leq d}\,.\] 
We start with $\Ff_\alpha^d$ and calculate using Lemmas \ref{lem: sl2invariantsubring} and \ref{lem: commutatorfsigma}:
\begin{align*}\Ff_\alpha^d(R_{k, m, l}\sigma_l)&=\big(\sigma_l\cdot  \Ff_\alpha^d+[\Ff_\alpha^d, \sigma_l]\big)R_{k, m, l}\\
&=\sigma_l\cdot \left( -\alpha\frac{\partial^2}{\partial\alpha^2}
    +(g+2d-1-l)\frac{\partial}{\partial \alpha}
    -2\gamma\frac{\partial^2}{\partial \alpha\partial\gamma}
    +\frac{1}{2}\beta\gamma\frac{\partial^2}{\partial\gamma^2}-\frac{g-l}{2}\beta\frac{\partial}{\partial\gamma}\right)R_{k, m, l}\\
    &=\left(\sum_{\substack{b+c=m\\a+b+2c=k-g-l}}(g-l-c)!\underbrace{(-a+1+g+2d-1-l-2c)}_{=2g+2d-k+b}\frac{\alpha^{a-1}}{(a-1)!}\frac{\beta^b}{b!}\frac{(2\gamma)^c}{c!}\right)\cdot \sigma_l\\
    &\quad+\left(\sum_{\substack{b+c=m\\a+b+2c=k-g-l}}\underbrace{(g-l-c)!(b+1)(c-1-g+l)}_{-(g-l-c+1)!(b+1)}\frac{\alpha^{a}}{a!}\frac{\beta^{b+1}}{(b+1)!}\frac{(2\gamma)^{c-1}}{(c-1)!}\right)\cdot \sigma_l\,.
\end{align*}
Renaming $(a-1, b, c)$ as $(a,b,c)$ in the first sum and $(a, b+1, c-1)$ as $(a,b,c)$ in the second sum, we find that
\[\Ff_\alpha^d(R_{k, m, l}\sigma_l)=\sum_{\substack{b+c=m\\a+b+2c=k-g-l-1}}(g-l-c)!(2g+2d-k)\frac{\alpha^a}{a!}\frac{\beta^b}{b!}\frac{(2\gamma)^c}{c!}\sigma_l=(2g+2d-k)R_{k-1, m, l}\sigma_l\,.\]
Now the conclusion that $\Ff_\alpha^d(R_{k, m, l}\sigma_l)\in I^{\textup{gr}}_{\leq d}$ for every $k\geq 2g+2d$ follows: if $k=2g+2d$ then the right hand side vanishes, and if $k>2g+2d$ then $k-1\geq 2g+2d$ as well, so $R_{k-1, m, l}\sigma_l\in I^{\textup{gr}}_{\leq d}$. 

The analysis for $\Ff_\beta^d$ is essentially parallel, and we obtain
\[\Ff_\beta^d(R_{k, m, l}\sigma_l)=\sum_{\substack{b+c=m-1\\a+b+2c=k-g-l-1}}(g-l-c)!(2g+2d-k)\frac{\alpha^a}{a!}\frac{\beta^b}{b!}\frac{(2\gamma)^c}{c!}\sigma_l=(2g+2d-k)R_{k-1, m-1, l}\sigma_l\,.\]
Hence, we conclude in the same way that $\Ff_\beta^d$ preserves the ideal $I^{\textup{gr}}_{\leq d}$. 
\end{proof}

\appendix

\section{Cohomology of the moduli stack with fixed determinant}

\label{appendix}

Fix $r\in\BZ_{\geq 1}$, $d\in \BZ$ and $\Lambda\in \Pic^d(\Sigma)$ throughout the appendix. All results in this appendix should be known to experts, although we could not find in the literature the exact form of the statements we need.

Recall that $\FN_{r,\Lambda}$ denotes the stack of rank $r$ vector bundles $U$ on $\Sigma$ together with an isomorphism $\phi:\det U\xrightarrow{\sim}\Lambda$. We first prove Proposition \ref{prop: cohomologiesofstacks} (i). The descendent algebra in the lemma below is the rank $r$ version of Definition \ref{def: descendents}.

\begin{lem}\label{lem: free generation for twisted}
    The realization map $\BD\to H^\ast(\FN_{r,\Lambda})$ is a ring isomorphism.
\end{lem}
\begin{proof}
    Note that $\FN_{r,\CO_\Sigma}$ is the stack of principal $\textnormal{SL}_r$-bundles on $\Sigma$. For such a stack, the classical result of \cite{AB} proves that the cohomology ring is freely generated by tautological generators which is precisely the statement of the lemma; see also \cite{HS} for an algebro-geometric proof. 

    Now we consider the general case. The proof of \cite[Proposition 3.1.1]{HS} applies to the ``$\Lambda$-twisted" case and shows that the realization map $\BD\to H^\ast(\FN_{r,\Lambda})$ is injective. To show that this is in fact an isomorphism, it suffices to show that the Poincaré series of $\FN_{r,\Lambda}$ and $\FN_{r,\CO_\Sigma}$ are equal. For this, it suffices to equate the Poincaré series of $\FN_{r,\Lambda}$ and $\FN_{r,\Lambda'}$ when $\Lambda\simeq \Lambda'\otimes \CO_\Sigma(x)$ for some $x\in \Sigma$. Consider the Hecke stack $\mathfrak{H}=\mathfrak{H}_{r,\Lambda,x}$ parametrizing non-split short exact sequences 
    $$0\rightarrow U'\rightarrow U\rightarrow \CO_x\rightarrow 0
$$
together with an isomorphism $\phi:\det(U)\xrightarrow{\sim} \Lambda$. Since there is a natural isomorphism $\det(U)\simeq \det(U')\otimes \CO_\Sigma(x)$, fixing the isomorphism $\phi:\det(U)\xrightarrow{\sim} \Lambda$ is equivalent to fixing the isomorphism $\phi':\det(U')\xrightarrow{\sim} \Lambda'$. The Hecke stack admits two forgetful morphisms
\begin{center}
    \begin{tikzcd}
  & \mathfrak{H} \arrow[ld, "\pi_1"'] \arrow[rd, "\pi_2"] &   \\
\FN_{r,\Lambda} &                                    & \FN_{r,\Lambda'}.
\end{tikzcd}
\end{center}
The fiber of $\pi_1$ at $(U,\phi)$ is naturally identified with the $(r-1)$-dimensional projective space $\BP(\Hom(U,\CO_x))$. The fiber of $\pi_2$ at $(U,\phi)$ is also given by the $(r-1)$-dimensional projective space $\BP(\Ext^1(\CO_x,U'))$. In fact, the Hecke stack is a projective bundle in two different ways
$$\BP_{\FN_{r,\Lambda}}(p_*\mathcal{H}om(\CU,\CO_x))\simeq\mathfrak{H}\simeq \BP_{\FN_{r,\Lambda'}}(R^1p_*\mathcal{H}om(\CO_x,\CU')). 
$$
By the projective bundle formula, the Poincaré series of $\FN_{r,\Lambda}$ and $\FN_{r,\Lambda'}$ are equal. 
\end{proof}

For a vector bundle $V$ of rank $r$ and degree $d$, denote by $t(V)=(r',d')$ the pair of rank and degree of the maximal destabilizing subbundle of $V$. The set of all possible values of $t(V)$ is
$$\mathsf{MD}_{r,d}\coloneqq \big\{(r,d)\}\sqcup \{(r',d')\,|\,0<r'<r,\ d'/r'>d/r\big\}.
$$ 
Define a total order $\leq$ on $\mathsf{MD}_{r,d}$ as follows; we say that $(r',d')<(r'',d'')$ if and only if
$$\big[\,d'/r'<d''/r''\,\big]\quad\textnormal{or}\quad\big[\,d'/r'=d''/r''\textnormal{ and } r'<r''\,\big].
$$
The order $\leq$ satisfies the property that there are only finitely many elements bounded above. In particular, it makes $\mathsf{MD}_{r,d}$ well-ordered, so we can apply the induction principle.\footnote{\label{footnote: maximal destabilizing and type}For each fixed $\mu\in \BQ$, let $t'=(r',d')\in \mathsf{MD}_{r,d}$ be the maximal element such that $d'/r'\leq \mu$. Then $t(V)\leq t'$ if and only if the slope of the maximal destabilizing subbundle of $V$ is $\leq \mu$, i.e., $\FM^{\leq \mu}_{r,d}=\FM^{\leq t'}_{r,d}$.
}
 
For each $t'=(r',d')\in \mathsf{MD}_{r,d}$, denote by $\FM_{r,d}^{\leq t'}$ (resp. $\FM_{r,d}^{< t'}$ and $\FM_{r,d}^{= t'}$) the stack of all vector bundles $V$ with $t(V)\leq t'$ (resp. $t(V)< t'$ and $t(V)= t'$). Since the maximal destabilizing type of a vector bundle only gets bigger under a specialization, $\FM_{r,d}^{\leq t'}\subseteq \FM_{r,d}$ is an open substack. Therefore, the full stack $\FM_{r,d}$ is stratified by the locally closed substacks
$$\FM_{r,d}=\bigsqcup_{t'\in \mathsf{MD}_{r,d}}\FM_{r,d}^{=t'}. 
$$
Let $j$ (resp. $\iota$) be the open (resp. closed) embedding induced from this stratification
\begin{equation}\label{eq: stratification}
    \FM^{<t'}_{r,d}\xlongrightarrow{j} \FM^{\leq t'}_{r,d} \xlongleftarrow{\iota} \FM^{=t'}_{r,d}.
\end{equation}
The same stratification exists also for the fixed determinant stack $\FN_{r,\Lambda}$. 

\begin{lem}\label{lem: split Gysin}
    The Gysin sequence associated to \eqref{eq: stratification} splits, i.e., we have a short exact sequence
    $$0\rightarrow H^{*-2\cdot\textnormal{codim}(i)}(\FM^{=t'}_{r,d})\xlongrightarrow{\iota_*} H^*(\FM^{\leq t'}_{r,d})\xlongrightarrow{j^*}H^*(\FM^{<t'}_{r,d})\rightarrow 0. 
    $$
    In particular, $H^*(\FM_{r,d}^{\leq t'})$ is generated by tautological classes. The same holds for $\FN_{r,\Lambda}^{\leq t'}$. 
\end{lem}
\begin{proof}
    The statements about $\FM_{r,d}^{\leq t'}$ are essentially in \cite[Proposition 3.5]{EK}, although loc. cit. is written in gauge-theoretic terms; see also \cite[Corollary 5.14]{H} for an algebro-geometric approach to a similar statement with respect to a different stratification. We explain the proof for $\FM_{r,d}^{\leq t'}$ for the reader's convenience and point out the necessary modifications for $\FN_{r,\Lambda}^{\leq t'}$. 

    Let $t'=(r',d')\neq (r,d)$. For each $V\in \FM_{r,d}^{=t'}$, we have a maximal destabilizing sequence
    $$0\rightarrow V'\rightarrow V\rightarrow V''\rightarrow 0
    $$
    where $V'$ is semistable of rank $r'$ and degree $d'$ and the quotient bundle $V''$ is of rank $r''\coloneqq r-r'$ and $d''\coloneqq d-d'$. It follows that the maximal destabilizing subbundle of $V''$ has a slope smaller than $\mu'\coloneqq d'/r'$. Let $t''\in\mathsf{MD}_{r'',d''}$ be the maximum element whose slope is smaller than $\mu'$. Such $t''$ exists because there are only finitely many elements in $\mathsf{MD}_{r'',d''}$ of slope bounded above. By the choice of $t''$, we have $V''\in\FM_{r'',d''}^{\leq t''}$. Conversely, given any $(V',V'')\in \FM_{r',d'}^{\textup{ss}}\times \FM_{r'',d''}^{\leq t''}$ and an extension $0\rightarrow V'\rightarrow V\rightarrow V''\rightarrow 0$, the extension bundle $V$ satisfies $t(V)=t'$. 

    The above discussion yields a diagram
    \begin{equation}\label{eq: diagram}
    \begin{tikzcd}
  & \FM_{r,d}^{=t'} \arrow[ld, "\pi"'] \arrow[rd, "i"] &   \\
\FM_{r',d'}^{\textup{ss}}\times \FM_{r'',d''}^{\leq t''} &                                    & \FM_{r,d}^{\leq t'}
\end{tikzcd}
    \end{equation}
where $\pi$ is the total space of a vector bundle stack such that
$$\pi^{-1}(V',V'')=[\Ext^1(V'',V')/\Hom(V'',V')].$$
On the other hand, the normal bundle of the closed  embedding $\iota$ is given by 
$$N_\iota=R^1p_*\big(\CH om(\CV', \CV'')\big)
$$
where $0\rightarrow \CV'\rightarrow \CV\rightarrow \CV''\rightarrow 0$ is the universal exact sequence over $\FM_{r,d}^{=t'}\times \Sigma$. This is in fact a bundle because $\Hom(V',V'')=0$ for any $(V',V'')\in \FM_{r',d'}^{\textup{ss}}\times \FM_{r'',d''}^{\leq t''}$. 

In order to show that the Gysin sequence associated to \eqref{eq: stratification} splits, it suffices to show that $\iota^*\iota_*=e(N_\iota)$ is injective. Since $N_\iota$ is naturally pulled back from from $\pi$, this reduces to showing that $e(N_\iota)\in H^*(\FM_{r',d'}^{\textup{ss}}\times \FM_{r'',d''}^{\leq t''})\simeq H^*(\FM_{r,d}^{=t'})$ is not a zero divisor. This follows from \cite[Lemma 5.13]{H} because $\FM_{r',d'}^{\textup{ss}}\times \FM_{r'',d''}^{\leq t''}$ is a $\BG_m$-gerbe over the rigidified stack $\FM_{r',d'}^{\textup{ss}}\times \Big(\FM_{r'',d''}^{\leq t''}\Big)^{\textup{rig}}$ \cite{ACV} with respect to which $N_\iota$ has weight $1$. Furthermore, tautological generation of $H^*(\FM_{r,d}^{\leq t'})$ follows from surjectivity of $j^*$, the tautological generation of the full stack $H^*(\FM_{r,d})$, and the fact that for any fixed $k$ the map $H^k(\FM_{r,d})\to H^k(\FM_{r,d}^{\leq t'})$ is an isomorphism for sufficiently large $t'$ (which is a consequence of the stacks $\FM_{r,d}^{=t'}$ having increasingly large codimension).

Now we explain the necessary modification for the stack $\FN^{\leq t'}$. We have the following diagram similar to \eqref{eq: diagram}:
    \begin{equation*}
    \begin{tikzcd}
  & \FN_{r,\Lambda}^{=t'} \arrow[ld, "\pi_\Lambda"'] \arrow[rd, "\iota_\Lambda"] &   \\
\FN_{r,\Lambda}^{\textnormal{split-}t'} &                                    & \FM_{r,\Lambda}^{\leq t'}
\end{tikzcd}
    \end{equation*}
where the splitting stack $\FN_{r,\Lambda}^{\textup{split-}t'}$ is defined by the Cartesian diagram 
\begin{center}
    \begin{tikzcd}
\FN^{\textnormal{split-}t'}_{r,\Lambda} \arrow[r] \arrow[d] & \displaystyle \FM_{r',d'}^{\textup{ss}}\times \FM_{r'',d''}^{\leq t''}\arrow[d] \arrow[r, phantom]& (V',V'') \arrow[d, phantom, "{\rotatebox[origin=c]{270}{$\mapsto$}}"] \\
\{\Lambda\} \arrow[r]           & \PPic^d  \arrow[r, phantom]& \displaystyle \det(V')\otimes \det(V'').      
    \end{tikzcd}
\end{center}
As in the varying determinant case, the morphism $\pi_\Lambda$ is a vector bundle stack defining an isomorphism on the cohomology rings. So we are left to prove that $e(N_{\iota_\Lambda})\in H^*(\FN_{r,\Lambda}^{\textnormal{split-}t'})$ is not a zero divisor. By definition of the splitting stack $\FN_{r,\Lambda}^{\textnormal{split-}t'}$, it has a generic stabilizer group 
$$T=\{(\lambda_1,\lambda_2)\in (\BC^*)^2\,|\,\lambda_1^{r'}\cdot \lambda_2^{r''}=1\}. 
$$
Set $\bar{r}'\coloneqq r'/\gcd(r',r'')$ and $\bar{r}''\coloneqq r''/\gcd(r',r'')$. Then elements inside $T$ of the form $(\lambda^{-\bar{r}''}, \lambda^{\bar{r}'})$ for $\lambda\in \BC^*$ forms a subgroup isomorphic to $\BC^*$. This realizes $\FN_{r,\Lambda}^{\textnormal{split-}t'}$ as a $\BG_m$-gerbe over the rigidified stack with respect to which the normal bundle $N_{\iota_\Lambda}$ has weight $\bar{r}'+\bar{r}''>0$. Therefore, the Gysin sequence of $\iota_\Lambda$ splits. The same argument shows that $H^*(\FN_{r,\Lambda}^{\leq t'})$ is tautologically generated starting from the statement about the full stack as in Lemma \ref{lem: free generation for twisted}. 
\end{proof}

Consider the tensoring map 
\begin{equation}\label{eq: tensoring map}
    \PPic^0\times \FN_{r,\Lambda}^{\leq t'}\rightarrow \FM_{r,d}^{\leq t'},\quad (L,U)\mapsto L\otimes U.
\end{equation}
The next lemma proves Proposition \ref{prop: cohomologiesofstacks} (ii). 

\begin{lem}
    The pullback under the tensoring map \eqref{eq: tensoring map} induces an isomorphism on the cohomology rings. 
\end{lem}
\begin{proof}

The tensoring map \eqref{eq: tensoring map} is an open restriction of the tensoring map
$$\PPic^0\times \FN_{r,\Lambda}\rightarrow \FM_{r,d}
\,.$$
When $\Lambda=\CO_\Sigma$, this map is obtained by applying the mapping stack $\textup{Map}(\Sigma,-)$ to the morphism $B\BC^*\times B\textup{SL}_r\twoheadrightarrow B\textup{GL}_r$. Since the kernel of $\BC^*\times \textup{SL}_r\twoheadrightarrow \textup{GL}_r$ is $\mu_r$, the fibers of the above tensoring map are isomorphic to
$$\textup{Map}(\Sigma,\mu_r)\simeq B\mu_r\times \Pic[r]
$$
where $\Pic[r]\subseteq \Pic$ denotes the $r$-torsion subgroup. 
More generally, we can factorize the tensoring map into 
$$\PPic^0\times \FN_{r,\Lambda}\xrightarrow{\ f\ } \PPic^0\times \FN_{r,\Lambda} \kern -5 pt \fatslash \kern -1.5 pt B\mu_r \xrightarrow{\ g\ } \FM_{r,d}
$$
where $f$ is a $\mu_r$-gerbe and $g$ is a quotient map by the $\Pic[r]$-action. This factorization follows from the fact that we have a short exact sequence
$$1\rightarrow \mu_r\rightarrow \Aut(L)\times \Aut(U,\phi)\rightarrow \Aut(L\otimes U)\rightarrow 1
$$
and any $V\in \FM_{r,d}$ can be written as $V=U\otimes L$ for some $L\in \PPic^0$ and $(U,\phi)\in \FN_{r,\Lambda}$ up to the ambiguity of $\Pic[r]$-action. Therefore, the pullback map also  factorizes into
$$H^*(\FM_{r,d})\xrightarrow{\ g^*\ }H^*(\PPic^0\times \FN_{r,\Lambda} \kern -5pt \fatslash \kern -1.5 pt B\mu_r)\xrightarrow{\ f^*\ }H^*(\PPic^0\times \FN_{r,\Lambda}). 
$$
Since $f$ is a gerbe with respect to a finite group $\mu_r$ and we use rational coefficients, $f^*$ is an isomorphism. On the other hand, $g^*$ identifies $H^*(\FM_{r,d})$ with the $\Pic[r]$-invariant part of the cohomology ring 
$$H^*(\PPic^0\times \FN_{r,\Lambda} \kern -5pt \fatslash \kern -1.5 pt B\mu_r)\simeq H^*(\PPic^0\times \FN_{r,\Lambda}).
$$
Since $\Pic[r]$ acts trivially on tautological classes and $\PPic^0\times \FN_{r,\Lambda}$ is tautologically generated, $g^*$ is an isomoprhism. The same factorization also exists for \eqref{eq: tensoring map}. Therefore the lemma follows from tautological generation of $\FN_{r,\Lambda}^{\leq t'}$ proven in Lemma \ref{lem: split Gysin}. 
\end{proof}

\end{document}